\documentclass[final,leqno]{siamltex}

\usepackage[algo2e,linesnumbered,vlined,ruled]{algorithm2e}
\usepackage{graphicx,graphics,subfigure}
\usepackage{url}
\usepackage{epsf,epstopdf}

\usepackage{comment}
\usepackage{times}
\usepackage{amsmath,amsxtra,amsfonts,amscd,amssymb,bm}
\usepackage[usenames]{color}
\usepackage[normalem]{ulem}
\usepackage{exscale}

\newtheorem{assumption}[theorem]{Assumption}

\newcommand{\be}{\begin{equation}}
\newcommand{\ee}{\end{equation}}
\newcommand{\bee}{\begin{equation*}}
\newcommand{\eee}{\end{equation*}}
\newcommand{\bea}{\begin{eqnarray}}
\newcommand{\eea}{\end{eqnarray}}
\newcommand{\beaa}{\begin{eqnarray*}}
\newcommand{\eeaa}{\end{eqnarray*}}

\newcommand{\minimize}{\mathop{\textrm{min}}}
\newcommand{\st}{\,\textrm{subject to}\,}  
\newcommand{\R}{\mathbb{R}}  
  
\newcommand{\cC}{\mathcal{C}}  
\newcommand{\cD}{\mathcal{D}}  
\newcommand{\cF}{\mathcal{F}}  
\newcommand{\cK}{\mathcal{K}}  
  
\newcommand{\cI}{\mathcal{I}}  
  
\newcommand{\cS}{{\mathcal{S}}}
\newcommand{\cO}{{\mathcal{O}}}
\newcommand{\cP}{{\mathcal{P}}}
  
\newcommand{\bA}{{D}}

\newcommand{\ind}{\mathrm{1}}

\newcommand{\prox}{\mathbf{prox}} 
\newcommand{\iprod}[2]{\left\langle{#1},{#2}\right\rangle}
 
\newcommand{\dist}{\mathrm{dist}} 
\newcommand{\sign}{\mathrm{sign}} 

\newcommand{\dct}{\mathrm{dct}}

\DeclareMathOperator*{\argmin}{arg\,min}

\allowdisplaybreaks
\graphicspath{ {./results/} }

\title{A Regularized Semi-Smooth Newton Method With Projection Steps for  Composite Convex Programs\thanks{The first version was titled ``Semi-Smooth Second-order Type Methods for Composite Convex Programs" (http://arxiv.org/abs/1603.07870). This version was submitted to ``SIAM Journal on Optimization" on July 28, 2016.}}
\author{
Xiantao Xiao\thanks{School of Mathematical Sciences, Dalian University of
Technology, Dalian, CHINA (xtxiao@dlut.edu.cn). Research supported by the Fundamental Research Funds for the Central Universities under the grant DUT16LK30.}
\and
 Yongfeng Li\thanks{School of Mathematical Sciences, Peking University, Beijing,
 CHINA (1200010733@pku.edu.cn).}  
 \and 
 Zaiwen Wen\thanks{Beijing International Center for Mathematical
Research, Peking University, Beijing, CHINA (wenzw@pku.edu.cn).
Research supported in part by NSFC grant 11322109 and by the National
Basic Research Project under the grant 2015CB856002.}
\and Liwei Zhang\thanks{School of Mathematical Sciences, Dalian University of Technology, Dalian, CHINA (lwzhang@dlut.edu.cn). Research partially supported by NSFC grants 11571059 and 91330206.}   
}

\begin{document}
\maketitle
\begin{abstract}
The goal of this paper is to study approaches to bridge the gap between first-order and second-order type methods for composite convex programs.  Our key observations are: i) Many well-known operator splitting methods, such as forward-backward splitting (FBS) and Douglas-Rachford splitting (DRS), actually define a fixed-point mapping;  ii) The optimal solutions of the composite convex program and the solutions of a system of nonlinear equations derived from the fixed-point mapping are equivalent. Solving this kind of system of nonlinear equations enables us to develop second-order type methods. Although these nonlinear equations may be non-differentiable, they are often semi-smooth and their generalized Jacobian matrix is positive semidefinite due to monotonicity.  By combining with a regularization approach and a known hyperplane projection technique, we propose an adaptive semi-smooth Newton method  and establish its convergence to global optimality. Preliminary numerical results on $\ell_1$-minimization problems demonstrate that our second-order type algorithms are able to achieve superlinear or quadratic convergence. 
\end{abstract} 
\section{Introduction}\label{sec:intro}
This paper aims to solve a composite convex optimization problem in the form
\be\label{eq:P}
\minimize\limits_{x\in\R^n}\quad f(x)+h(x),
\ee
where $f$ and $h$ are extended real-valued convex functions. Problem (\ref{eq:P}) arises
from a wide variety of applications, such as signal recovery, image processing,
machine learning, data analysis, and etc. For example, it becomes the sparse
optimization problem when $f$ or $h$ equals to the $\ell_1$-norm, which attracts a
significant interest in signal or image processing in recent years. If $f$ is a
loss function associated with linear predictors and $h$ is a regularization
function, problem (\ref{eq:P}) is often referred as the regularized 
risk minimization problem  in machine learning and
statistics.  When $f$ or $h$ is an indicator function onto a convex set, problem
(\ref{eq:P}) represents a general convex constrained optimization problem.

Recently, a series of first-order methods, including the forward-backward
splitting (FBS) (also known as proximal gradient) methods, Nesterov's accelerated methods, the alternative
direction methods of
multipliers (ADMM), the Douglas-Rachford splitting (DRS) and Peaceman-Rachford
splitting (PRS) methods, have been extensively studied
and  widely used for solving a subset of problem (\ref{eq:P}). The readers are
referred to, for example, \cite{BPCPE2011,CP2011} and references therein, for a
review on some of these first-order methods. One main feature of these methods
is that they first  exploit the underlying problem structures,  then construct
subproblems that can be solved relatively efficiently. These algorithms are rather
simple yet powerful since they are easy to be implemented in many interested
applications and they often converge fast to a solution with moderate accuracy.
However,  a notorious drawback is that they may suffer from a slow tail
convergence and hence a significantly large number of iterations is needed in order
to achieve a high accuracy. 

A few Newton-type methods for some special instances of problem (\ref{eq:P}) have been
investigated to alleviate the inherent weakness of the first-order type methods.
Most existing Newton-type methods for problem (\ref{eq:P}) with
a differentiable function $f$ and a simple function $h$ whose proximal mapping
can be cheaply evaluated are based on the FBS method to some
extent. The proximal Newton method \cite{LSS2014,PSB2014} can be interpreted as a generalization of the proximal gradient
method. It updates in each iteration by a composition of the proximal mapping
with a Newton or quasi-Newton step. The semi-smooth Newton methods
proposed in \cite{GL2008,MU2014,BCNO2015} solve the nonsmooth
formulation of the optimality conditions corresponding to the FBS method. 
 In \cite{ZST2010}, the augmented Lagrangian method is applied to solve the dual
 formulation of general linear semidefinite
programming problems, where each augmented Lagrangian function is minimized by
using  the semi-smooth Newton-CG method. Similarly, a proximal point algorithm is
 developed to solve the dual problems of a class of matrix spectral norm approximation
in \cite{CLST2016}, where the subproblems are again handled by the semi-smooth Newton-CG method.

In this paper, we study a few second-order type methods for problem
(\ref{eq:P}) in a general setting even if $f$ is nonsmooth and $h$ is an indicator
function. Our key observations are that many first-order methods, such as
the FBS and DRS methods, can be written as fixed-point iterations and the
optimal solutions of \eqref{eq:P} are also the solutions of a system of
nonlinear equations  defined by the corresponding fixed-point mapping. Consequently, the concept is to develop second-order type 
algorithms based on solving the
system of nonlinear equations. Although these nonlinear equations are often
non-differentiable,  they are monotone and can be semi-smooth due to the properties of
the proximal mappings.
 We first propose a regularized semi-smooth Newton method to solve the system
 of nonlinear equations. The regularization term is important since the generalized Jacobian
 matrix corresponding to monotone equations may only be positive semidefinite.
In particular, the regularization parameter is updated by a self-adaptive strategy similar to
the trust region algorithms.
  By combining with the semi-smooth Newton step and a hyperplane projection technique,
  we show that the method converges globally to an optimal solution of problem
  (\ref{eq:P}). The hyperplane projection step is in fact indispensable for the
  convergence to global optimality and it is inspired by several iterative methods for solving monotone nonlinear equations
  \cite{SS1999,ZL2008}.  Different from the approaches in the literature, the
  hyperplane projection step is only executed when the residual of the semi-smooth
  Newton step is not reduced sufficiently. When certain conditions are
  satisfied, we prove that the semi-smooth
  Newton steps are always performed close to the optimal solutions. Consequently,
  fast local convergence rate is established. 
  For some cases, the computational cost can be further reduced if the Jacobian
  matrix is approximated by the limited memory BFGS (L-BFGS) method. 

Our main contribution is the study of some relationships between the  first-order
and second-order type methods. Our semi-smooth Newton methods are able to solve the general convex composite
problem (\ref{eq:P})  as long as a fixed-point mapping is well defined.
In particular,  our methods are applicable to  constrained convex programs, such
as constrained $\ell_1$-minimization problem. In
contrast, the Newton-type methods in \cite{LSS2014,PSB2014,GL2008,MU2014,BCNO2015}
are designed for unconstrained problems. Unlike the methods in
\cite{ZST2010,CLST2016} applying the semi-Newton method to
 a sequence of subproblems, our target is a single system
of nonlinear equations. Although solving
the Newton system is a major challenge, the computational cost usually can be
controlled reasonably well when certain structures can be utilized.
Our preliminary numerical results show that
our proposed methods are able to reach  superlinear or quadratic convergence
rates on typical $\ell_1$-minimization problems. 


The  rest of this paper is organized as follows.  In section \ref{sec:split}, we
review a few popular operator splitting methods, derive their equivalent
fixed-point iterations and state their convergence properties.   We propose a semi-smooth
Newton method and establish its  convergence results in section \ref{sec:newton}.
Numerical results on a number of  applications are presented in section \ref{sec:apps}. Finally, we conclude this paper in section \ref{sec:conclusion}.

\subsection{Notations}
 Let $I$ be the identity operator or identity matrix of suitable size. Given a
 convex function $f:\R^n\rightarrow(-\infty,+\infty]$ and a scalar $t>0$, the \textit{proximal mapping} of $f$ is defined by
\be\label{eq:prox}
\prox_{tf}(x):=\argmin\limits_{u\in\R^n} f(u)+\frac{1}{2t}\|u-x\|_2^2.
\ee
If $f(x)=\ind_\Omega(x)$ is the indicator function of a nonempty closed convex set $\Omega\subset\R^n$, then the proximal mapping $\prox_{tf}$ reduces to the \textit{metric projection} defined by
\be\label{eq:proj}
\cP_\Omega(x):=\argmin\limits_{u\in \Omega} \frac{1}{2}\|u-x\|_2^2.
\ee
The Fenchel conjugate function $f^*$ of $f$ is 
\be\label{eq:conjugate}
f^*(y):=\sup\limits_{x\in\R^n}\{x^Ty-f(x)\}.
\ee
A function $f$ is said to be \textit{closed} if its epigraph is closed, or
equivalently $f$ is lower semicontinuous. 
A mapping $F:\R^n\rightarrow\R^n$ is said to be \textit{monotone}, if 
\[\iprod{x-y}{F(x)-F(y)}\geq 0,\quad \mbox{for all}\ x,y\in\R^n.
\]
\section{Operator splitting and fixed-point algorithms}\label{sec:split}
This section reviews some  operator splitting algorithms for problem
(\ref{eq:P}), including FBS, DRS, and ADMM. These algorithms are well studied in the literature, see \cite{FP2003II, BC2011, CP2011,DY2014I} for example. Most of the operator splitting algorithms can also be interpreted as fixed-point algorithms derived from certain optimality conditions.
\subsection{FBS}
In problem (\ref{eq:P}), let $h$ be a continuously differentiable function. The FBS algorithm is the iteration
\be\label{eq:FBS}
x^{k+1}=\prox_{tf}(x^k-t\nabla h(x^k)),\quad k=0,1,\ldots,
\ee
where $t>0$ is the step size. 
When $f(x)=\ind_C(x)$ is the indicator function of a closed convex set $C$, FBS reduces to the \textit{projected gradient method} for solving the constrained program
\[
\minimize\limits_{x\in \R^n}\quad h(x)\quad
\st\quad  x\in C.
\]
Define the following operator
\be\label{eq:f-FBS}
T_{\textrm{FBS}}:=\prox_{tf}\circ(I-t\nabla h).
\ee
Then FBS can be viewed as a fixed-point iteration 
\be\label{eq:fix-FBS}
x^{k+1}=T_{\textrm{FBS}}(x^k).
\ee

\subsection{DRS}
The DRS algorithm solves (\ref{eq:P}) by the following update:
\begin{align}
&x^{k+1} = \prox_{th}(z^k)\label{eq:DR1},\\
&y^{k+1} = \prox_{tf}(2x^{k+1}-z^k)\label{eq:DR2},\\
&z^{k+1} = z^k+y^{k+1}-x^{k+1}\label{eq:DR3}.
\end{align}
The algorithm is traced back to  \cite{DR1956,LM1979,EB1992} to solve partial differential equations (PDEs).
 The fixed-point iteration characterization of DRS is in the form of
\be\label{eq:fix-DRS}
z^{k+1}=T_{\textrm{DRS}}(z^k),
\ee
where
\be\label{eq:f-DRS}
T_{\textrm{DRS}}:=I+\prox_{tf}\circ(2\prox_{th}-I)-\prox_{th}.
\ee
  

\subsection{Dual operator splitting and ADMM}
Consider a linear constrained program:
\be\label{eq:LCP}
\begin{array}{ll}
\minimize
\limits_{x_1\in \R^{n_1},x_2\in\R^{n_2}}\quad &f_1(x_1)+f_2(x_2)\\
\st\quad & A_1x_1+A_2x_2=b,
\end{array}
\ee
where $A_1\in\R^{m\times n_1}$ and $A_2\in\R^{m\times n_2}$. The dual problem of (\ref{eq:LCP}) is given by
\be\label{eq:DLCP}
\minimize\limits_{w\in\R^m}\quad d_1(w)+d_2(w),
\ee
where 
\[
d_1(w):=f_1^*(A_1^Tw),\quad d_2(w):=f_2^*(A_2^Tw)-b^Tw.
\]

Assume that $f_1$ is closed and  strongly convex (which implies that $\nabla d_1$ is Lipschitz \cite[Proposition 12.60]{RW1998}) and $f_2$ is convex.
The FBS iteration for the dual problem (\ref{eq:DLCP})  can be expressed in
terms of the variables in the original problem under the name  alternating minimization algorithm, which is also equivalent to the fixed-point iteration
\[
w^{k+1}=T_{\textrm{FBS}}(w^k).
\]

Assume that $f_1$ and $f_2$ are convex. It is widely known that the DRS iteration for dual problem (\ref{eq:DLCP}) is
the ADMM \cite{GM1975,GM1976}. It is regarded as a variant of augmented
Lagrangian method and has attracted much attention in numerous fields. A recent
survey paper \cite{BPCPE2011} describes the applications of the ADMM to
statistics and machine learning. The ADMM is equivalent to the following fixed-point iteration
\[
z^{k+1}=T_{\textrm{DRS}}(z^k),
\]
where $T_{\textrm{DRS}}$ is the DRS fixed-point mapping for problem (\ref{eq:DLCP}).
\subsection{Convergence of the fixed-point algorithms}
We summarize the relationship between the aforementioned fixed-points and the optimal solution of problem (\ref{eq:P}), and review the existing convergence results on the fixed-point algorithms.

The following lemma is straightforward, and its proof is omitted.
\begin{lemma}\label{lem:relation}
Let the fixed-point mappings $T_{\textrm{FBS}}$ and $T_{\textrm{DRS}}$  be defined in (\ref{eq:f-FBS}) and (\ref{eq:f-DRS}), respectively.
\begin{itemize}
\item[(i)] Suppose that $f$ is closed, proper and convex, and $h$ is convex and continuously differentiable. A fixed-point of $T_{\textrm{FBS}}$ is equivalent to an optimal solution to problem (\ref{eq:P}). 
\item[(ii)] Suppose that $f$ and $h$ are both closed, proper and convex. Let $z^*$ be a fixed-point of $T_{\textrm{DRS}}$, then $\prox_{th}(z^*)$ is an optimal solution to problem (\ref{eq:P}).
\end{itemize}
\end{lemma}


Error bound condition is a useful property for establishing the linear
convergence of a class of first-order  methods including the FBS method and
ADMM, see \cite{FP2003I, LT1993, Tseng2010, HSZ2015} and the references therein.
Let $X^*$ be the optimal solution set of problem (\ref{eq:P}) and $F(x)\in \R^n$ be a  
residual function satisfying $F(x) = 0$ if and only if $x \in X^*$. The
definition of error bound condition is given as follows.
\begin{definition}\label{def:EBC}
The error bound condition  holds for some test set $T$ and some residual function $F(x)$ if there exists a constant $\kappa > 0$ such that
\be\label{eq:EBC}
\dist(x, X^*) \leq \kappa \|F(x)\|_2 \quad  \mbox{ for all } x \in T.
\ee
In particular,
it is said that error bound condition with residual-based test set ($\emph{EBR}$) holds if the test set in (\ref{eq:EBC}) is selected by $T:=\{x\in\R^n|f(x)+h(x)\leq v, \|F(x)\|_2\leq\varepsilon\}$ for some constant $\varepsilon\geq 0$ and any $v \geq v^*:=\min_x f(x)+h(x)$.
\end{definition}

Under the error bound condition, the fixed-point iteration of FBS is proved to converge linearly, see  \cite[Theorem 3.2]{DL2016} for example.
\begin{proposition}[Linear convergence of FBS]\label{prop:LR-FBS}
Suppose that error bound condition (EBR) holds with parameter $\kappa$ for residual function $F_{\textrm{FBS}}$. Let $x^*$ be the limit point of the sequence $\{x^k\}$ generated by the fixed-point iteration $x^{k+1}=T_{\textrm{FBS}}(x^k)$ with $t\leq\beta^{-1}$ for some constant $\beta>0$. Then there exists an index $r$ such that for all $k\geq 1$,
\[
\|x^{r+k}-x^*\|_2^2\leq \left(1-\frac{1}{2\kappa\beta}\right)^kC\cdot(f(x^r)+h(x^r)-f(x^*)-h(x^*)),
\]
where $C:=\frac{2}{\beta(1-\sqrt{1-(2\beta\gamma)^{-1}})^2}$.
\end{proposition}

%
Finally, we mention that the sublinear convergence rate of some general
fixed-point iterations has been well studied, see \cite[Theorem 1]{DY2014I}. 

\section{Semi-smooth Newton method for nonlinear monotone equations}\label{sec:newton}

The purpose of this section is to design a Newton-type method for solving the system of nonlinear equations
\be\label{eq:monoNE}
F(z)=0,
\ee
where $F:\R^n\rightarrow\R^n$ is strongly semi-smooth and monotone. In
particular, we are interested in $F(z)=z-T(z)$, where $T(z)$ is a fixed-point
mapping corresponding to certain first-order type algorithms.

\subsection{Semi-smoothness of proximal mapping}\label{sec:prox}
We now discuss the semi-smoothness of proximal mappings. This property often
implies that the fixed-point mappings corresponding to operator splitting algorithms are semi-smooth or strongly semi-smooth.

Let $\cO\subseteq\R^n$ be an open set and $F:\cO\rightarrow\R^m$
be a locally Lipschitz continuous function.  Rademacher's theorem says that $F$
is almost everywhere differentiable. Let $D_F$ be the set of differentiable
points of $F$ in $\cO$. We next introduce the concepts of generalized differential.
\begin{definition}\label{def:GD}
Let $F:\cO\rightarrow\R^m$ be locally Lipschitz continuous at $x\in\cO$. The \emph{B-subdifferential} of $F$ at $x$ is defined by
\[
\partial_BF(x):=\left\{\lim\limits_{k\rightarrow\infty}F'(x^k)| x^k\in D_F, x^k\rightarrow x\right\}.
\]
The set 
$\partial F(x)=\emph{co}(\partial_BF(x))$
is called Clarke's \emph{generalized Jacobian}, where $\emph{co}$ denotes the convex hull.
\end{definition}

The notion of semi-smoothness plays a key role on establishing locally superlinear convergence of the nonsmooth Newton-type method. Semi-smoothness was originally introduced by Mifflin \cite{Mifflin1977} for real-valued functions and extended to vector-valued mappings by Qi and Sun \cite{QS1993}.
\begin{definition}\label{def:semi}
Let $F:\cO\rightarrow\R^m$ be a locally Lipschitz continuous function. We say that $F$ is  semi-smooth at $x\in\cO$ if
\begin{itemize}
\item[(a)] $F$ is directionally differentiable at $x$; and
\item[(b)]  for any $d\in\cO$ and $J\in\partial F(x+d)$,
\[
\|F(x+d)-F(x)-Jd\|_2=o(\|d\|_2)\quad\mbox{as}\ d\rightarrow 0.
\]
\end{itemize}
Furthermore, $F$ is said to be strongly  semi-smooth at $x\in\cO$ if $F$ is semi-smooth and for any $d\in\cO$ and $J\in\partial F(x+d)$,
\[
\|F(x+d)-F(x)-Jd\|_2=O(\|d\|_2^2)\quad\mbox{as}\ d\rightarrow 0.
\]
\end{definition}

(Strongly) semi-smoothness is closed under scalar multiplication, summation and composition. The examples of semi-smooth functions include the smooth functions, all convex functions (thus norm), and the piecewise differentiable functions. Differentiable functions with Lipschitz gradients are strongly semi-smooth. For every $p\in[1,\infty]$, the norm $\|\cdot\|_{p}$ is strongly semi-smooth. Piecewise affine functions are strongly semi-smooth, such as $[x]_+=\max\{0,x\}$. A vector-valued function is (strongly) semi-smooth if and only if each of its component functions is (strongly) semi-smooth.  Examples of semi-smooth functions are thoroughly studied in \cite{FP2003II, Ulbrich2011}.

The basic properties of proximal mapping is well documented in textbooks such as
\cite{RW1998, BC2011}. The proximal mapping $\prox_{f}$, corresponding to a
proper, closed and convex function $f:\R^n\rightarrow\R$, is single-valued,
maximal monotone and nonexpansive. Moreover, the proximal mappings of many interesting functions
are (strongly) semi-smooth. It is worth mentioning that the semi-smoothness of
proximal mapping does not hold in general \cite{Shapiro1994}. The following lemma is useful when the proximal mapping of a function is complicate but the proximal mapping of its conjugate is easy.
\begin{lemma}[Moreau's decomposition]\label{lem:deco}
Let $f:\R^n\rightarrow\R$ be  a proper, closed and convex function. Then, for any $t>0$ and $x\in\R^n$,
\[
x=\prox_{tf}(x)+t\prox_{f^*/t}(x/t).
\]
\end{lemma}

We next review some existing results on the semi-smoothness of proximal mappings of various interesting functions.
The proximal mapping of $\ell_1$-norm $\|x\|_1$, which is the well-known soft-thresholding operator, is component-wise separable and piecewise affine. Hence, the operator $\prox_{\|\cdot\|_1}$ is strongly semi-smooth. According to the Moreau's decomposition,
the proximal mapping of $\ell_{\infty}$ norm (the conjugate of $\ell_1$ norm) is also strongly semi-smooth.
For $k\in\mathbb{N}$,  a function with $k$ continuous derivatives is called a $\cC^k$ function.
A function $f:\cO\rightarrow\R^m$ defined on the open set $\cO\subseteq\R^n$ is
called piecewise $\cC^k$  function, $k\in[1,\infty]$, if $f$ is continuous and
if at every point $\bar{x}\in\cO$ there exists a neighborhood $V\subset\cO$ and a finite collection of $\cC^k$ functions $f_i:V\rightarrow\R^m, i=1,\ldots,N$, such that
\[
f(x)\in\{f_1(x),\ldots,f_N(x)\}\quad \textrm{for all}\ x\in V.
\]
For a comprehensive study on piecewise $\cC^k$ functions, the readers are
referred to \cite{Scholtes2012}. From \cite[Proposition 2.26]{Ulbrich2011}, if $f$ is a piecewise $\cC^1$ (piecewise smooth) function, then $f$ is semi-smooth; if $f$ is a piecewise $\cC^2$ function, then $f$ is strongly semi-smooth.  As described in \cite[Section 5]{PSB2014}, in many applications the proximal mappings are piecewise $\cC^1$ and thus semi-smooth.
Metric projection, which is the proximal mapping of an indicator function, plays
an important role in the analysis of constrained programs. The projection over a
polyhedral set is piecewise linear \cite[Example 12.31]{RW1998} and hence
strongly semi-smooth. The projections over symmetric cones are proved to be
strongly semi-smooth in \cite{SS2002}. 

\subsection{Monotonicity of fixed-point mappings}\label{subsec:monotone}
This subsection focuses on the discussion of the monotonicity of the fixed-point mapping $F:=I-T$, where $T:\R^n\rightarrow\R^n$ is a fixed-point operator. Later, we will show that the monotone property of $F$ plays a critical role in our proposed method.

For the sake of readability, let us first recall some related concepts.  A mapping $F:\R^n\rightarrow\R^n$ is called \textit{strongly monotone} with modulus $c>0$ if  
 \[\iprod{x-y}{F(x)-F(y)}\geq c\|x-y\|_2^2,\quad\mbox{for all}\ x,y\in\R^n.
 \] 
It is said that $F$ is \textit{cocoercive} with modulus $\beta>0$ if 
\[\iprod{x-y}{F(x)-F(y)}\geq \beta\|F(x)-F(y)\|_2^2,\quad \mbox{for all}\ x,y\in\R^n.
\]

We now present the monotone properties of the fixed-point mappings $F_{\textrm{FBS}}=I-T_{\textrm{FBS}}$ and $F_{\textrm{DRS}}=I-T_{\textrm{DRS}}$.

\begin{proposition}\label{prop:mono-FBS}
	\begin{itemize}
		\item[(i)] Suppose that $\nabla h$ is cocoercive with $\beta>0$, then $F_{\textrm{FBS}}$ is monotone if  $0<t\leq 2\beta$.
		\item[(ii)] Suppose that $\nabla h$ is strongly monotone with $c>0$ and Lipschitz with $L>0$, then $F_{\textrm{FBS}}$ is strongly monotone if  $0<t<2c/L^2$.
		\item[(iii)]	 Suppose that $h\in C^2$, $H(x):=\nabla^2h(x)$ is positive semidefinite for any $x\in\R^n$ and $\bar{\lambda}=\max_x\lambda_{\textrm{max}}(H(x))<\infty$. Then, $F_{\textrm{FBS}}$ is monotone  if $0<t\leq 2/\bar{\lambda}$.
	    \item[(iv)] 	The fixed-point mapping $F_{\textrm{DRS}}:=I-T_{\textrm{DRS}}$  is monotone.
	\end{itemize}
\end{proposition}
\begin{proof}
Items (i) and (ii) are well known in the literature, see \cite{ZL2001} for example.
%
		
(iii) From the mean value theorem, there exists some $x'$ such that
\[
\nabla h(x)-\nabla h(y)=H(x')(x-y).
\]
Hence,
$ \|\nabla h(x)-\nabla h(y)\|^2
\leq\bar{\lambda}\iprod{x-y}{\nabla h(x)-\nabla h(y)}$, 
which implies that $\nabla h$ is cocoercive with $1/\bar{\lambda}$. Hence, the monotonicity is obtained from item (i) .

	(iv) It has been shown that the operator $T_{\textrm{DRS}}$ is firmly nonexpansive, see \cite{LM1979}. Therefore, $F_{\textrm{DRS}}$ is firmly nonexpansive  and hence monotone \cite[Proposition 4.2]{BC2011}.
\end{proof}

Items (i) and (ii) demonstrate that $F_{\textrm{FBS}}$ is monotone as long as the step size $t$ is properly selected.
It is also shown in \cite{ZL2001} that, when $f$ is an indicator function of a
convex closed set, the step size interval in items (i) and (ii)  can be enlarged to $(0,4\beta]$ and $(0,4c/L^2)$, respectively. Item (iii)  can also be found in \cite[Lemma 4.1]{HYZ2008}.
Finally, we introduce an useful lemma on the positive semidefinite property of the subdifferential of the monotone mapping.
\begin{lemma}\label{lem:sdp-mono}
	For a monotone and Lipschitz continuous mapping $F:\R^n\rightarrow\R^n$ and any $x\in\R^n$, each element of $\partial_BF(x)$ is positive semidefinite.
\end{lemma}
\begin{proof} We first show that $F'(\bar{x})$ is positive semidefinite at a differentiable point $\bar{x}$. Suppose that there exist constant $a>0$ and $d\in\R^n$ with $\|d\|_2=1$ such that $\iprod{d}{F'(\bar{x})d}=-a$. For any $t>0$, let
$	\Phi(t):=F(\bar{x}+td)-F(\bar{x})-tF'(\bar{x})d$.  
	Since $F$ is differentiable at $\bar{x}$, we have $\|\Phi(t)\|_2=o(t)$ as $t\rightarrow 0$. The monotonicity of $F$ indicates that
	\[
	\begin{array}{ll}
			0&\leq\iprod{td}{F(\bar{x}+td)-F(\bar{x})}=\iprod{td}{tF'(\bar{x})d+\Phi(t)}\\[5pt]
			&\leq -at^2+t\|d\|_2\|\Phi(t)\|_2= -at^2+o(t^2),
	\end{array}
	\]
	which leads to a contradictory.
	
	For any $x\in\R^n$ and each $J\in\partial_BF(x)$, there exists a sequence of
    differentiable points $x^k\rightarrow x$ such that $F'(x^k)\rightarrow J$. Since every $F'(x^k)$ is positive semidefinite, we have that $J$ is also positive semidefinite.
\end{proof}

\subsection{A Regularized semi-smooth Newton method with Projection steps}
The system of monotone equations has various applications
\cite{OR1970,SS1999,ZL2008,LL2011,AAB2013}.  Inspired by a pioneer work \cite{SS1999},  a class of iterative
methods for solving nonlinear (smooth) monotone equations were proposed in
recent years \cite{ZL2008,LL2011,AAB2013}. In \cite{SS1999}, the authors
proposed a globally convergent Newton method by exploiting the structure of
monotonicity, whose primary advantage is that the whole sequence of the
distances from the iterates to the solution set  is decreasing. 
The method is extended in \cite{ZT2005} to solve monotone equations without nonsingularity assumption. 

The main concept in \cite{SS1999} is introduced as follows. For an iterate $z^k$, let $d^k$ be a descent direction such that
\[
\iprod{F(u^k)}{-d^k}>0,
\]
where $u^k=z^k+d^k$ is an intermediate iterate.
By monotonicity of $F$, for any $z^*\in Z^*$ one has
\[
\iprod{F(u^k)}{z^*-u^k}\leq 0.
\]
Therefore, the hyperplane
\[
H_k:=\{z\in\R^n|\iprod{F(u^k)}{z-u^k}=0\}
\]
strictly separates $z^k$ from the solution set $Z^*$. 
 Based on this fact, it was developed in \cite{SS1999} that the next iterate is set by
\[
z^{k+1}=z^k-\frac{\iprod{F(u^k)}{z^k-u^k}}{\|F(u^k)\|_2^2}F(u^k).
\]
It is easy to show that the point $z^{k+1}$ is the projection of $z^k$ onto the
hyperplane $H_k$.  The hyperplane projection step is  critical  to construct a
globally convergent method for solving the system of nonlinear monotone equations. 
By applying the same technique, we develop a globally convergent
method for solving semi-smooth monotone equations (\ref{eq:monoNE}).

It has been demonstrated in Lemma \ref{lem:sdp-mono} that each element of the B-subdifferential of a monotone and
semi-smooth mapping is positive semidefinite. Hence, for an iterate $z^k$, by
choosing an element $J_k\in\partial_BF(z^k)$, it is natural to apply a
regularized Newton method. It computes
\be\label{eq:r-newton}
(J_k+\mu_kI)d=-F^k,
\ee
where $F^k=F(z^k)$, $\mu_k=\lambda_k\|F^k\|_2$ and $\lambda_k>0$ is a
regularization parameter. The regularization term $\mu_k I$ is chosen such that
$J_k+\mu_k I$ is invertible. From a computational view, it is practical to
solve the linear system (\ref{eq:r-newton}) inexactly.
Define
\be\label{eq:rk}
r^k:=(J_k+\mu_kI)d^k+F^k.
\ee
At each iteration, we seek a step $d^k$ by solving (\ref{eq:r-newton}) approximately such that
\be\label{eq:residual}
\|r^k\|_2\leq\tau\min\{1,\lambda_k\|F^k\|_2 \|d^k\|_2\},
\ee
where $0<\tau<1$  is some positive constant. Then a trial point is obtained as
\[
u^k=z^k+d^k.
\]
Define a ratio
\be\label{eq:rho}
\rho_k=\frac{-\iprod{F(u^k)}{d^k}}{\|d^k\|_2^2}.
\ee
Select some parameters $0<\eta_1\le\eta_2<1$ and $1<\gamma_1 <\gamma_2$. If
$\rho_k\geq\eta_1$, the iteration is said to be \textit{successful}.
Otherwise, the iteration is \textit{unsuccessful}. Moreover, for a successful
iteration, if $\|F(u^k)\|_2$ is sufficiently decreased, we take a Newton step,
otherwise we take a hyperplane projection step. In summary, we set
\be\label{eq:zk}
z^{k+1}=\left\{\begin{array}{lll}
u^k,\ &\textrm{if}\ \rho_k\geq\eta_1 \mbox{ and } \|F(u^k)\|_2 \le \nu \|F(\bar{u})\|_2, &\mbox{ [Newton step]}\\
v^k,\ &\textrm{if}\
\rho_k\geq\eta_1 \mbox{ and } \|F(u^k)\|_2 > \nu \|F(\bar{u})\|_2, &\mbox{ [projection step]}\\
z^k, &\textrm{otherwise}. &\mbox{ [unsuccessful iteration]}
\end{array}\right.
\ee
where $0<\nu<1$, 
\be\label{eq:HP-step}
v^k=z^k-\frac{\iprod{F(u^k)}{z^k-u^k}}{\|F(u^k)\|_2^2}F(u^k),
\ee
 and the reference point $\bar{u}$ is the iteration from the last Newton step. More specifically, when $\rho_k\geq\eta_1$  and  $\|F(u^k)\|_2 \le \nu \|F(\bar{u})\|_2$, we take $z^{k+1}=u^k$ and update $\bar{u}=u^k$.

Finally, the regularization parameter $\lambda_k$ is updated as
\begin{equation}\label{eq:lambda}
\lambda_{k+1}\in\left\{
\begin{array}{lll}
(\uline{\lambda},\lambda_k),\quad &\textrm{if}\ \rho_k\geq\eta_2,\\[8pt]
[\lambda_k,\gamma_1\lambda_k],\quad &\textrm{if}\ \eta_1\leq\rho_k<\eta_2,\\[8pt]
(\gamma_1\lambda_k,\gamma_2\lambda_k],\quad &\textrm{otherwise,}
\end{array}
\right.
\end{equation}
where $\uline{\lambda}>0$ is a small positive constant.
These parameters determine how aggressively the regularization parameter is
decreased when an iteration is successful or it is increased when an iteration
is unsuccessful. The complete approach to solve (\ref{eq:monoNE}) is summarized
in Algorithm \ref{algo:newton}.

\begin{algorithm2e}[htb!]\label{algo:newton}\caption{\emph{An Adaptive Semi-smooth Newton
  (ASSN) method}}
 Give  $0<\tau,\nu<1$,  $0<\eta_1\leq\eta_2<1$ and
  $1<\gamma_1\leq\gamma_2$ \; Choose $z^0$ and $\varepsilon>0$. Set $k=0$ and
  $\bar{u}=z^0$ \;
  \While{not ``converged''}{
Select $J_k\in\partial_B F(x^k)$\;
Solve the linear system (\ref{eq:r-newton})
approximately such that $d^k$ satisfies (\ref{eq:residual}) \;
 Compute $u^k=z^k+d^k$ and calculate the ratio $\rho_k$ as in (\ref{eq:rho}) \;
Update $z^{k+1}$ and $\lambda_{k+1}$ according to (\ref{eq:zk}) and
(\ref{eq:lambda}), respectively \; 
  Set $k=k+1$\;
  }
\end{algorithm2e}

\subsection{Global convergence}
It is clear that  a solution is obtained if Algorithm \ref{algo:newton} terminates in finitely many iterations. Therefore, we assume that Algorithm \ref{algo:newton} always generates an infinite sequence $\{z^k\}$ and $d^k\neq 0$ for any $k\geq 0$.  Let $Z^*$ be
the solution set of system (\ref{eq:monoNE}). Throughout this section, we assume that $Z^*$ is nonempty.
The following assumption is used in the sequel.
\begin{assumption}\label{assu:jacobian}
  Assume that $F:\R^n\rightarrow\R^n$ is strongly semi-smooth and monotone.
Suppose that there exists a constant $c_1>0$ such that 
$\|J_k\|\leq c_1$ for any $k\geq 0$ and any $J_k\in\partial_{B}F(z^k)$.
\end{assumption}

The following lemma demonstrates that the distance from  $z^k$ to $Z^*$
decreases in a projection step. The proof follows directly from \cite[Lemma
2.1]{SS1999}, and it is omitted.
\begin{lemma}\label{lem:hyper}
For any $z^*\in Z^*$ and any projection step, indexed by say $k$, we have that
\be\label{eq:distance}
\|z^{k+1}-z^*\|_2^2\leq\|z^k-z^*\|_2^2-\|z^{k+1}-z^k\|_2^2.
\ee
\end{lemma}

Recall that $F$ is strongly semi-smooth. Then for a point $z\in\R^n$ there exists $c_2>0$ (dependent on $z$) such that for any $d\in\R^n$ and any $J\in\partial_BF(z+d)$  ,
\be\label{eq:semi}
\|F(z+d)-F(z)-Jd\|_2\leq c_2\|d\|_2^2,\quad\textrm{as}\ \|d\|_2\rightarrow 0.
\ee
Denote the index sets of  Newton steps, projection steps and successful iterations, respectively,   by
\[
\begin{array}{ll}
\cK_{N}:=\{k\geq 0: \rho_k\geq\eta_1,\  \|F(u^k)\| \le \nu \|F(\bar{u})\|\}, \\[5pt]
\cK_{P}:=\{k\geq 0: \rho_k\geq\eta_1,\  \|F(u^k)\| > \nu \|F(\bar{u})\|\}
\end{array}
\]
and
\[
\cK_{S}:=\{k\geq 0: \rho_k\geq\eta_1\}.
\]
We next show that if there are only finitely many successful iterations, the later iterates are optimal solutions.
\begin{lemma}\label{lem:finite}
Suppose that Assumption \ref{assu:jacobian} holds and the  index set $\cK_{S}$ is finite. Then $z^k=z^*$ for all sufficiently large $k$ and $F(z^*)=0$.
\end{lemma}
\begin{proof}
Denote the index of the last successful iteration by $k_0$. The construction of
the algorithm implies that $z^{k_0+i}=z^{k_0+1}:=z^*$, for all $i\geq 1$ and
additionally $\lambda_k\rightarrow\infty$. Suppose that $a:=\|F(z^*)\|_2>0$. For
all $k> k_0$, it follows from (\ref{eq:rk}) that  
\[
d^k=(J_k+\lambda_k\|F^k\|_2I)^{-1}(r^k-F^k),
\]
which, together with $\lambda_k\rightarrow\infty$, $\|r^k\|_2\leq\tau$ and the fact that $J_k$ is positive semidefinite, imply that $d^k\rightarrow 0$, and hence $u^k\rightarrow z^*$.

We now show that when  $\lambda_k$ is large enough, the ratio $\rho_k$ is not
smaller than $\eta_2$. For this purpose, we consider an iteration with index  $k> k_0$ sufficiently large such that $\|d^k\|_2\leq 1$ and
\[
\lambda_k\geq\frac{\eta_2+c_1+c_2}{a-\tau a}.
\]
Then, it yields that
\be\label{eq:ineq-a1}
\begin{array}{ll}
-\iprod{F(z^k)}{d^k}&=\iprod{(J_k+\lambda_k\|F^k\|_2I)d^k)}{d^k}-\iprod{r^k}{d^k}\\[5pt]
&\geq\lambda_k\|F^k\|_2\|d^k\|_2^2-\tau\lambda_k\|F^k\|_2\|d^k\|_2^2\\[5pt]
&\geq(\eta_2+c_1+c_2)\|d^k\|_2^2.
\end{array}
\ee
Further, for any $J_{u^k}\in\partial_BF(u^k)$ we obtain
\[
\begin{array}{ll}
-\iprod{F(u^k)}{d^k}\\
=-\iprod{F(z^k)}{d^k}-\iprod{J_{u^k}d^k}{d^k}+\iprod{-F(u^k)+F(z^k)+J_{u^k}d^k}{d^k}\\[5pt]
\geq-\iprod{F(z^k)}{d^k}-c_1\|d^k\|_2^2-\|F(z^*+d^k)-F(z^*)-J_{u^k}d^k\|_2\\[5pt]
\geq(\eta_2+c_1+c_2)\|d^k\|_2^2-c_1\|d^k\|_2^2-c_2\|d^k\|_2^2\\[5pt]
=\eta_2\|d^k\|_2^2,
\end{array}
\]
 where the first inequality is from Assumption \ref{assu:jacobian} and the facts that
 $\|d^k\|_2\leq 1$ and $z^k=z^*$, and the second inequality comes from (\ref{eq:ineq-a1}) and
 (\ref{eq:semi}). Hence, we have $\rho_k\geq\eta_2$, which generates a successful iteration and yields a contradiction. This completes the proof.
\end{proof}

The following result shows that a solution is derived if the set $\cK_{N}$ is infinite.
\begin{lemma}\label{lem:inf-KN}
Let Assumption \ref{assu:jacobian} hold. If the sequence $\{z^k\}$ contains infinitely many iterates resulting from Newton steps, i.e., $|\cK_{N}|=\infty$, then $\{z^k\}$ converges to some point $\bar{z}$ such that $F(\bar{z})=0$.	
\end{lemma}
\begin{proof}
We first show that a subsequence of $\{z^k\}$ converges to a solution of $F(z)=0$.
	Let $(k_i)_{i\geq 0}$ enumerate all elements of the set $\{k+1:k\in\cK_{N}\}$ in increasing order. Since $\|F(z^{k_i})\|_2\leq \nu\|F(z^{k_{i-1}})\|_2$ and $0<\nu<1$, we have that the subsequence $\{z^{k_i}\}$ converges to a solution $\bar{z}$ as $i\rightarrow\infty$. 
	
	For any $k\notin\cK_{N}$, we have $\|z^{k+1}-\bar{z}\|_2\leq\|z^{k}-\bar{z}\|_2$ from the updating rule (\ref{eq:zk}) and Lemma \ref{lem:hyper}. Moreover, for any $k\notin\cK_{N}$, there exists an index $i$ such that $k_i<k+1<k_{i+1}$, and hence $\|z^{k+1}-\bar{z}\|_2\leq\|z^{k_i}-\bar{z}\|_2$. Therefore, the whole sequence $\{z^k\}$ converges to $\bar{z}$.
\end{proof}

We are now ready to prove the main global convergence result. In specific, we show that the infinite sequence $\{z^k\}$ generated by Algorithm \ref{algo:newton} always converges to some solution.
\begin{theorem}\label{th:newton}
Let Assumption \ref{assu:jacobian} hold. Then $\{z^k\}$ converges to some point $\bar{z}$ such that $F(\bar{z})=0$.
\end{theorem}
\begin{proof}
If  the index set $\cK_{S}$ is finite, the result is directly from Lemma \ref{lem:finite}. The case that $\cK_N$ is infinite has bee established in Lemma \ref{lem:inf-KN}. The remaining part of the proof is to deal with the occurrence of that $\cK_N$ is finite and $\cK_P$ is infinite. In this situation, without loss of generality, we can ignore $\cK_N$ and assume that $\cK_S=\cK_P$ in the sequel.

Let $z^*$ be any point in solution set $Z^*$.  By Lemma \ref{lem:hyper}, for any $k\in 
\cK_{S}$, it yields that 
\be\label{eq:dis1}
\|z^{k+1}-z^*\|_2^2\leq \|z^k-z^*\|_2^2-\|z^{k+1}-z^k\|_2^2.
\ee
Therefore, the sequence $\{\|z^k-z^*\|_2\}$ is non-increasing and convergent,  the sequence $\{z^k\}$ is bounded, and
\be\label{eq:dis2}
\lim\limits_{k\rightarrow\infty}\|z^{k+1}-z^k\|_2=0.
\ee
By (\ref{eq:rk}) and (\ref{eq:residual}), it follows that
\[
\|F^k\|_2\geq\|(J_k+\lambda_k\|F^k\|_2I)d^k\|_2-\|r^k\|_2\geq(1-\tau)\lambda_k\|F^k\|_2\|d^k\|_2,
\]
which implies that $\|d^k\|_2\leq 1/[(1-\tau)\uline{\lambda}]$.
This inequality shows that $\{d^k\}$ is bounded, and $\{u^k\}$ is also bounded. By using the continuity of $F$, there exists a constant $c_3>0$ such that
\[
\|F(u^k)\|_2^{-1}\geq c_3,\quad \textrm{for any }\ k\geq 0.
\]
Using (\ref{eq:zk}),  for any $k\in\cK_{S}$, we obtain that
\[
\|z^{k+1}-z^k\|_2=\frac{-\iprod{F(u^k)}{d^k}}{\|F(u^k)\|_2}\geq c_3\rho_k\|d^k\|_2^2,
\]
which, together with (\ref{eq:dis2}),  imply that
\be\label{eq:conv1}
\lim\limits_{k\rightarrow\infty, k\in\cK_{S}}\rho_k\|d^k\|_2^2=0.
\ee

We next consider two possible cases:
\[
\liminf\limits_{k\rightarrow\infty}\|F^k\|_2=0\quad\textrm{and}\quad\liminf\limits_{k\rightarrow\infty}\|F^k\|_2=c_4>0.
\]
In the first case, the continuity of $F$ and the boundedness of $\{z^k\}$ imply
that the sequence $\{z^k\}$ has some accumulation point $\hat{z}$ such that
$F(\hat{z})=0$.  Since $z^*$ is an arbitrary point in $Z^*$, we can choose
$z^*=\hat{z}$ in (\ref{eq:dis1}). Then $\{z^k\}$ converges to $\hat{z}$.

In the second case, by using the continuity of $F$ and the boundedness of $\{z^k\}$ again, there exist constants $c_5>c_6>0$ such that
\[
c_6\leq\|F^k\|_2\leq c_5,\quad \textrm{for all}\ k\geq 0. 
\]
If $\lambda_k$ is large enough such that $\|d^k\|_2\leq 1$ and 
$$\lambda_k\geq \frac{\eta_2+c_1+c_2}{(1-\tau) c_6},$$
then by a similar proof as in Lemma \ref{lem:finite} we have that
$\rho_k\geq\eta_2$ and consequently $\lambda_{k+1}<\lambda_k$. Hence, it turns
out that $\{\lambda_k\}$ is bounded from above, by say $\bar{\lambda}>0$.
Using (\ref{eq:rk}), (\ref{eq:residual}), Assumption \ref{assu:jacobian} and the upper bound of $\{\lambda_k\}$, we have
\[
\|F^k\|_2\leq\|(J_k+\lambda_k\|F^k\|_2I)d^k\|_2+\|r^k\|_2\leq (c_1+(1+\tau) c_5\bar{\lambda})\|d^k\|_2.
\]
Hence, it follows that
\[
\liminf\limits_{k\rightarrow\infty}\|d^k\|_2>0.
\]
Then, by (\ref{eq:conv1}), it must hold that
\[
\lim\limits_{k\rightarrow\infty, k\in\cK_{S}}\rho_k=0,
\]
which yields a contradiction to the definition of $\cK_{S}$. Hence the second case is not possible.  The proof is completed.
\end{proof}

As is already shown,  the global convergence of  our Algorithm is essentially
guaranteed by the projection step. However, by noticing that (\ref{eq:HP-step})
is in the form of $v^k=z^k-\alpha_kF(u^k)$ with
$\alpha_k=\iprod{F(u^k)}{z^k-u^k}/\|F(u^k)\|_2^2>0$, the projection step is
indeed an extragradient step \cite{FP2003II}. Since the  asymptotic convergence
rate of the extragradient step is often not faster than that of the Newton step, a
slow convergence may be observed if the projection step is always performed.
Hence, our modification \eqref{eq:zk} is practically meaningful. Moreover, we
will next prove that the projection step will never be performed when the iterate is close
enough to a solution under some generalized nonsingular conditions.

\subsection{Fast local convergence}
Since Algorithm \ref{algo:newton} has been shown to be globally convergent, we
now assume that the sequence $\{z^k\}$ generated by Algorithm \ref{algo:newton}
converges to a solution $z^*\in Z^*$. Under some reasonable conditions, we will
prove that the Newton steps achieve a locally quadratic convergence. Moreover, we
will show that when the iteration point $z^k$ is close enough to $z^*$, the
condition $\|F(u^k)\|_2\leq\nu\|F(z^k)\|_2$ is always satisfied. Consequently, Algorithm \ref{algo:newton}  turns into a second-order Newton method in a neighborhood of $z^*$.

We make the following assumption.
\begin{assumption}\label{assu:local}
The mapping $F$ is BD-regular at $z^*$, that is, all elements in $\partial_BF(z^*)$ are nonsingular.
\end{assumption}

The BD-regularity is a common assumption in the analysis of the local
convergence  of nonsmooth methods. The following properties of the BD-regularity are directly derived from \cite[Proposition 2.5]{Qi1993} and \cite[Proposition 3]{PQ1993}.

\begin{lemma}\label{lem:bd}
	Suppose that $F$ is BD-regular at $z^*$, then there exist  constants $c_0>0,\kappa>0$ and a neighborhood $N(z^*,\varepsilon_0)$  such that for any $y\in N(z^*,\varepsilon_0)$ and $J\in\partial_BF(y)$, 
	\begin{itemize}
		\item[(i)] $J$ is nonsingular and $\|J^{-1}\|\leq c_0$;
		\item[(ii)] $z^*$ is an isolated solution;
		\item[(iii)] the error bound condition holds for  $F(z)$ and $N(z^*,\varepsilon_0)$, that is $\|y-z^*\|_2\leq\kappa\|F(y)\|_2$.
	\end{itemize}
\end{lemma}
Since $z^*$ is isolated, the term $\dist(y,Z^*)$ in Definition \ref{def:EBC}
 is degenerated to $\|y-z^*\|_2$ and it becomes the error bound condition in item (iii).   The local convergence relies on some auxiliary results.
\begin{lemma}\label{lem:aux}
Suppose that Assumption \ref{assu:local} holds true, then
\begin{itemize}
	\item[(i)] the parameter $\lambda_k$ is bounded above by some constant $\bar{\lambda}>0$;
	\item[(ii)] there exists some $L> 0$ such that $\|F(z)\|_2\leq L\|z-z^*\|_2$ for any $z\in N(z^*,\varepsilon_0)$;
	\item[(iii)] for any $z^k\in N(z^*,\varepsilon_1)$ with $\varepsilon_1:=\min\{\varepsilon_0,1/(2Lc_0\tau\bar{\lambda})\}$, we have $$\|d^k\|_2\leq 2c_0L\|z^k-z^*\|_2.$$
\end{itemize}	
\end{lemma}
\begin{proof}
Item (i) has been shown in the proof of global convergence. The local Lipschitz continuity in item (ii) is obvious since $F$ is semi-smooth. For any $z^k\in N(z^*,\varepsilon_1)$, one has $\|F^k\|_2\leq L\|z^k-z^*\|_2\leq L\varepsilon_1$, hence 
\be\label{eq:aux1}
c_0\tau\lambda_k\|F^k\|_2\leq c_0\tau\bar{\lambda}\|F^k\|_2\leq 1/2.
\ee
 Note that
\[
\begin{array}{ll}
\|d^k\|_2&\leq \|(J_k+\mu_k I)^{-1}F^k\|_2+\|(J_k+\mu_k I)^{-1}r^k\|_2\\	
&\leq c_0L\|z^k-z^*\|_2+c_0\tau\lambda_k\|F^k\|_2\|d^k\|_2,
\end{array}
\]
we have $(1-c_0\tau\lambda_k\|F^k\|_2)\|d^k\|_2\leq c_0L\|z^k-z^*\|_2$, which, together with (\ref{eq:aux1}), yields $\|d^k\|_2\leq 2c_0L\|z^k-z^*\|_2$.
\end{proof}

We next show that the Newton steps are locally quadratically convergent.
\begin{theorem}\label{th:quad}
Suppose that Assumption \ref{assu:local} holds.  Then for any $k\in\cS_N$ and
$z^k\in N(z^*,\varepsilon_1)$, we have
\be\label{eq:quad}
\|z^{k+1}-z^*\|_2\leq c_7\|z^k-z^*\|_2^2,
\ee
where the constant $c_7:=c_0(c_2+(1+2c_0L\tau)\bar{\lambda}L)$.
\end{theorem}
\begin{proof}
For a Newton step, we have
\[
\begin{array}{ll}
	\|z^{k+1}-z^*\|_2=\|z^k+d^k-z^*\|_2\\[4pt]
	=\|z^k+(J_k+\mu_kI)^{-1}(F^k+(J_k+\mu_kI)d^k-F^k)-z^*\|_2\\[4pt]
	\leq\|z^k-z^*-(J_k+\mu_kI)^{-1}F^k\|_2+\|(J_k+\mu_k)^{-1}\|\cdot
\|F^k+(J_k+\mu_kI)d^k\|_2\\[4pt]
    \leq\|(J_k+\mu_k I)^{-1}\|\cdot(\|F^k-F(z^*)-J_k(z^k-z^*)\|_2+\mu_k\|z^k-z^*\|_2+\|r^k\|_2)\\[4pt]
        \leq\|J_k^{-1}\|\cdot(\|F^k-F(z^*)-J_k(z^k-z^*)\|_2+\lambda_k\|F^k\|_2\|z^k-z^*\|_2+\tau\lambda_k\|F^k\|_2\|d^k\|_2)\\[4pt]
    \leq\|J_k^{-1}\|\cdot(c_2\|z^k-z^*\|_2^2+\lambda_k\|F^k\|_2\|z^k-z^*\|_2+\tau\lambda_k\|F^k\|_2\|d^k\|_2)\\[4pt]
    \leq c_0(c_2\|z^k-z^*\|_2^2+(1+2c_0L\tau)\lambda_k\|F^k\|_2\|z^k-z^*\|_2)\\[4pt]
    \leq c_0(c_2+(1+2c_0L\tau)\bar{\lambda}L)\|z^k-z^*\|_2^2,
\end{array}
\]	
where the third inequality is from the facts that $\mu_k=\lambda_k\|F^k\|_2$ and
$\|r^k\|_2\leq\tau\lambda_k\|F^k\|_2\|d^k\|_2$,  the fourth inequality uses
(\ref{eq:semi}), and the fifth inequality arises from item (iii) in Lemma \ref{lem:aux}.
\end{proof}

Based on Theorem \ref{th:quad}, a region is defined in the following corollary.
It is shown that, $\|F(u^k)\|_2\leq\nu\|F^k\|_2$ is always satisfied in this
region. 
\begin{corollary}\label{coro:region}
Under the conditions of Theorem \ref{th:quad}, for any $z^k\in N(z^*,\varepsilon_2)$ with $\varepsilon_2:=\min\{\varepsilon_1,\nu/(Lc_7\kappa)\}$, we have $\|F(u^k)\|_2\leq\nu\|F(z^k)\|_2$.	
\end{corollary}
\begin{proof}
Using the Lipschitz of $F$, Theorem \ref{th:quad} and item (iii) in Lemma
\ref{lem:bd}, we obtain
\[
\|F(u^k)\|_2\leq L\|z^k+d^k-z^*\|_2\leq Lc_7\|z^k-z^*\|_2^2\leq Lc_7\varepsilon_2\|z^k-z^*\|_2\leq \nu\|F(z^k)\|_2.
\]	
\end{proof}


It is clear that the BD-regular condition plays a key role in the above
discussion. Although the BD-regular condition is strong and may fail in some
situations, there are some possible ways to resolve this  issue. As is shown in
\cite[Section 4.2]{MU2014}, suppose that there exists a nonsingular element in
$\partial_BF(z^*)$ and other elements in $\partial_BF(z^*)$ may be singular. By
exploiting the  structure of $\partial_BF(z)$,   one can carefully choose a
nonsingular generalized Jacobian when $z$ is close enough to $z^*$. Hence, if
$z^*$ is isolated,  one can still obtain the fast local convergence results by a
similar proof as above.  Another way is inspired by the literature on
the Levenberg-Marquardt (LM) method. The LM method is a regularized Gauss-Newton
method to deal with some possibly singular systems. It has been shown in \cite{FY2005}  that
the LM method preserves a superlinear or quadratic  local
convergence rate under certain local error bound condition, which is weaker than the nonsingular
condition. Therefore, it remains a future research topic to investigate 
 local convergence of our algorithm under the local error bound condition.

 \subsection{Regularized L-BFGS method with projection steps}
In this subsection, we propose a regularized L-BFGS method with projection steps by
simply replace the Newton step in Algorithm \ref{algo:newton} with a regularized
L-BFGS step to avoid solving the linear system (\ref{eq:r-newton}). The L-BFGS
method is an adaption of the classical BFGS method, which tries to use a minimal
storage.  A globally convergent BFGS method with projection steps is proposed in
\cite{ZL2008} for solving smooth monotone equations. The convergence of our
regularized L-BFGS method  can be analyzed in a similar way as  our regularized
Newton method by combining the convergence analysis in \cite{ZL2008}. We
only describe the L-BFGS update in the following and omit the convergence analysis.

For an iterate $z^k$, we  compute the direction   by 
\be\label{eq:bfgsd}
(H_k + \mu_k I)d^k = -F^k, 
\ee
where  $H_k$ is the L-BFGS approximation to the Jacobian matrix.
 
Choosing an initial matrix $H_k^{0}$ and setting $\delta F^k = F^{k+1} - F^{k}$,
the Jacobian matrix can be approximated by the recent $m$ pairs $\{\delta F^i,
d^i\}, i = k-1, k -2, \cdots, k-m$, i.e., using the standard formula
\cite{NocedalWright06} as
\be\label{eq:bfgs}
H_{k} = H_k^{0} - 
\left[
\begin{matrix}
H_k^0\cD_k & \cF_k
\end{matrix}\right]
\left[
\begin{matrix}
\cD_k^TH_k^0\cD_k & L_k \\
L^T_k & -S_k
\end{matrix}
\right]^{-1}
\left[
\begin{matrix}
\cD_k^T (H_k^0)^T \\
 \cF_k^T
\end{matrix}\right],
\ee
where $\cD_k = [d^{k-m}, \cdots, d^{k-1} ]$, $\cF_k = [\delta F^{k-m}, \cdots,  \delta F^{k-1} ]$, $L_k$ is a lower-triangular matrix with entries
\[
(L_k)_{i,j} = 
\begin{cases}\label{eq:BFGS}
(d^{k-m-1+i})^T(\delta F^{k-m-1+j})&  \text{if} \quad i > j, \\
0 &\text{otherwise},
\end{cases}
\]
and $S_k$ is a diagonal matrix with entries
\[
(S_k)_{ii} = (d^{k-m-1+i})^T\delta F^{k-m-1+i}.
\]
Then we can compute the inverse regularized Jacobian matrix
\[
(H_k + \mu_k I)^{-1} = \bar{H}_k^{-1} + \bar{H}_k^{-1}C_kR_k^{-1}C_k^T(\bar{H}_k^T)^{-1},
\]
where $\bar{H}_k = H_k^0 + \mu_k I$, $C_k = [H_k^0\cD_k  \quad \cF_k]$, $R_k$ is defined by $R_k = V_k - C_k^T \bar{H}_k^{-1} C_k$ and 
\[
V_k = 
\left[
\begin{matrix}
	\cD_{k}^TH_k^0\cD_{k} & L_{k} \\
	L_{k}^T & -S_{k}
\end{matrix}
\right].
\]
Specifically, if $k$ is smaller than $m$, we use the classical BFGS method to
approximate inverse regularized Jacobian matrix, which just let $d^j = \delta
F^j= 0$ for $j < 0$ in the formula \eqref{eq:bfgs}. 

%
%
%
%

\section{Numerical Results}\label{sec:apps}
In this section, we conduct proof-of-concept numerical experiments on our proposed
 schemes for the fixed-point mappings induced from the FBS and DRS methods by
 applying them to $\ell_1$-norm minimization problem. 
All numerical experiments are performed in {\sc Matlab} 
on workstation with a  Intel(R) Xeon(R) CPU E5-2680 v3 and 128GB memory.

\subsection{Applications to the FBS method} 
Consider the $\ell_1$-regularized optimization problem of the form
\be \label{eq:log} \min \; \mu \|x\|_1 + h(x), \ee
where $h$ is continuously differentiable.  
Let  $f(x)=\mu \|x\|_1$.  The system of nonlinear
equations
corresponding to the FBS method is 
$ F(x) =x- \prox_{tf}(x- t \nabla h(x))=0.$
The generalized Jacobian matrix of $F(x)$ is 
\be J(x) =I- M(x)(I- t \partial^2 h(x)) ,\ee
where $M(x)\in \partial \prox_{tf}( x - t \nabla h(x))$ and $\partial^2 h(x)$ is
the generalized Hessian matrix of $h(x)$.
 Specifically, the proximal mapping corresponding to $f(x)$ is the so-called
shrinkage operator defined as
\[ \left( \prox_{tf}(x) \right)_i= \sign(x_i) \max(|x_i|-\mu t, 0).
\]
Hence, one can take a Jacobian matrix $M(x)$ which is a diagonal matrix with diagonal entries being
\[ \left(M(x) \right)_{ii}=
\begin{cases} 1, & \mbox{ if } |(x-t\nabla h(x))_i| >\mu t,\\
  0, & \mbox{ otherwise. }
  \end{cases}\]
Similar to \cite{MU2014}, we introduce the index sets
\[
\begin{split}
&\mathcal{I}(x) := \{i : |(x - t \nabla h(x))_i| >t\mu  \}  = \{i: \left(M(x) \right)_{ii} = 1\}, \\
&\mathcal{O}(x) := \{i : |(x - t \nabla h(x))_i| \leq t\mu  \}  = \{i:
\left(M(x) \right)_{ii} = 0\}.
\end{split}
\]
The Jacobian matrix can be represented by
\[
J(x) = 
\left(
\begin{matrix}
t(\partial^2h(x))_{\mathcal{I}(x)\mathcal{I}(x)} & t(\partial^2h(x))_{\mathcal{I}(x)\mathcal{O}(x)} \\
0 & I 
\end{matrix}
\right).
\]
Using the above special structure of Jacobian matrix $J(x)$, we can reduce the
complexity of the regularized Newton step \eqref{eq:r-newton}. Let $\mathcal{I}
= \cI(x^k)$ and $\cO = \cO(x^k)$. Then,  we have 
\[\begin{split}
&(1 + \mu_k)s^k_{\cO} = -  F_{k, \cO} ,\\
&(t(\partial^2h(x))_{\cI\cI} + \mu I) s^k_{\cI} +  t(\partial^2h(x))_{\cI\cO} s^k_{\cO} = - F_{k,\cI},
\end{split}\]
which yields
\[\begin{split}
&s^k_{\cO} = -  \frac{1}{1 + \mu_k}F_{k, \cO}, \\
&(t(\partial^2h(x))_{\cI\cI} + \mu I) s^k_{\cI}  = - F_{k,\cI} - t(\partial^2h(x))_{\cI\cO} s^k_{\cO} .
\end{split}\]

\subsubsection{Numerical comparison} 
In this subsection, we compare our proposed methods with different solvers for
solving problem (\ref{eq:log}) with $h(x)=\frac{1}{2} \|Ax-b\|_2^2$. The solvers
used for comparison include ASSN, SSNP, ALSB, FPC-AS \cite{WYGZ2010}, SpaRSA
\cite{WNF2009} and SNF \cite{MU2014}. ASSN is the proposed semi-smooth Newton
method  with projection steps (Algorithm \ref{algo:newton} ) and SSNP is the
method which only uses the projection steps. ASLB(i) is a variant of the line
search based method by combining the L-BFGS method and hyperplane projection
technique. The number in bracket is the size of memory. FPC-AS is a first-order
method that uses a fixed-point iteration under Barzilai-Borwein steps and
continuation strategy. SpaRSA resembles FPC-AS, which is also a first-order
methods and uses Barzilai-Borwein steps and continuation strategy. SNF is a
semi-smooth Newton type method which uses the filter strategy and is one of
state-of-the-art second-order methods for $\ell_1$-regularized optimization problem
\eqref{eq:log} and SNF(aCG) is the SNF solver with an adaptive parameter strategy in the
conjugate gradient method. The parameters of FPC-AS, SpaRSA and SNF are the same as \cite{MU2014}.

The test problems are from \cite{MU2014}, which are constructed as follows. Firstly, we randomly generate a sparse solution $\bar{x} \in \mathbb{R}^n$ with $k$ nonzero entries, where $n = 512^2 = 262144$ and $k = [n/40] = 5553$. The $k$ different indices are uniformly chosen from $\{1,2,\cdots,n\}$ and we set the magnitude of each nonzero element by $\bar{x}_i = \eta_1(i)10^{d \eta_2(i)/20}$, where $\eta_1(i)$ is randomly chosen from $\{-1,1\}$ with probability $1/2$, respectively, $\eta_2(i)$ is uniformly distributed in $[0,1]$ and $d$ is a dynamic range which can influence the efficiency of the solvers. Then we choose $m = n/8 = 32768$ random cosine measurements, i.e., $Ax = (\dct(x))_J$, where $J$ contains $m$ different indices randomly chosen form  $\{1,2,\cdots,n\}$ and $\dct$ is the discrete cosine transform. Finally, the input data is specified by $b = A\bar{x} + \epsilon$, where $\epsilon$ is a Gaussian noise with a standard deviation $\bar{\sigma} = 0.1$. 

To compare fairly, we set an uniform stopping criterion. 
For a certain tolerance $\epsilon$, we obtain a solution $x_{newt}$ using ASSN such that $\|F(x_{newt})\|\le \epsilon$. Then we terminate all methods by the relative criterion
\[
\frac{f(x^k) - f(x^{*})}{\max\{f(x^*),1\}} \leq \frac{f(x_{newt}) -
f(x^{*})}{\max\{f(x^*),1\}},
\]
where $f(x)$ is the objective function and $x^{*}$ is a highly accurate solution obtained by ASSN under the criterion $||F(x)|| \leq 10^{-14}$.

\begin{table}[htb]
\caption{Total number of $A$- and $A^T$-  calls $N_A$ and CPU time (in seconds) averaged over 10 independent runs with dynamic range 20 dB}\label{tab:db20}
\setlength{\tabcolsep}{5pt}
\centering
\begin{tabular}{ccccccccccc}
\hline
method & \multicolumn{2}{c}{$\epsilon : 10^{0}$} &  \multicolumn{2}{c}{$\epsilon : 10^{-1}$} &  \multicolumn{2}{c}{$\epsilon : 10^{-2}$} & \multicolumn{2}{c}{$\epsilon : 10^{-4}$} &
\multicolumn{2}{c}{$\epsilon : 10^{-6}$}  \\ 
 \cline{2-3}  \cline{4-5}  \cline{6-7}   \cline{8-9}   \cline{10-11}  
& time & $N_A$ & time & $N_A$ & time & $N_A$ &  time & $N_A$ & time & $N_A$ \\ \hline
SNF& 1.12 & 84.6 & 2.62 & 205 & 3.19 & 254.2 & 3.87 & 307 & 4.5 & 351   \\ 
SNF(aCG)& 1.11 & 84.6 & 2.61 & 205 & 3.19 & 254.2 & 4.19 & 331.2 & 4.3 & 351.2   \\ 
ASSN& 1.15 & 89.8 & 1.81 & 145 & 2.2 & 173 & 3.15 & 246.4 & 3.76 & 298.2   \\ 
SSNP& 2.52 & 199 & 5.68 & 455.6 & 8.05 & 649.4 & 20.7 & 1679.8 & 29.2 & 2369.6   \\ 
ASLB(2)& 0.803 & 57 & 1.35 & 98.4 & 1.66 & 121 & 2.79 & 202.4 & 3.63 & 264.6   \\ 
ASLB(1)& 0.586 & 42.2 & 1.01 & 71.6 & 1.29 & 92 & 2.54 & 181.4 & 3.85 & 275   \\ 
FPC-AS& 1.45 & 109.8 & 5.03 & 366 & 7.08 & 510.4 & 10 & 719.8 & 10.3 & 743.6   \\ 
SpaRSA& 5.46 & 517.2 & 5.54 & 519.2 & 5.9 & 539.8 & 6.75 & 627 & 9.05 & 844.4   \\ 
\hline
\end{tabular}
\end{table}

\begin{table}[htb]
\caption{Total number of $A$- and $A^T$-  calls $N_A$ and CPU time (in seconds) averaged over 10 independent runs with dynamic range 40 dB}\label{tab:db40}
\setlength{\tabcolsep}{5pt}
\centering
\begin{tabular}{ccccccccccc}
\hline
method & \multicolumn{2}{c}{$\epsilon : 10^{0}$} &  \multicolumn{2}{c}{$\epsilon : 10^{-1}$} &  \multicolumn{2}{c}{$\epsilon : 10^{-2}$} & \multicolumn{2}{c}{$\epsilon : 10^{-4}$} &
\multicolumn{2}{c}{$\epsilon : 10^{-6}$}  \\ 
 \cline{2-3}  \cline{4-5}  \cline{6-7}   \cline{8-9}   \cline{10-11}  
& time & $N_A$ & time & $N_A$ & time & $N_A$ &  time & $N_A$ & time & $N_A$ \\ \hline
SNF& 2.12 & 158.2 & 4.85 & 380.8 & 6.07 & 483.2 & 6.8 & 525 & 7.2 & 562.4   \\ 
SNF(aCG)& 2.07 & 158.2 & 4.84 & 380.8 & 6.1 & 483.2 & 7.1 & 553.6 & 7.22 & 573.6   \\ 
ASSN& 2.34 & 182.2 & 3.67 & 285.4 & 4.29 & 338.6 & 5.11 & 407 & 5.92 & 459.2   \\ 
SSNP& 6.05 & 485.6 & 12.3 & 978.6 & 19.5 & 1606.6 & 27.3 & 2190.8 & 37.1 & 2952.2   \\ 
ASLB(2)& 1.39 & 98.2 & 2.19 & 154.4 & 2.64 & 194 & 3.45 & 250.4 & 4.49 & 323.6   \\ 
ASLB(1)& 1.25 & 86.8 & 1.84 & 127.4 & 2.2 & 161.6 & 3.2 & 225.6 & 4.59 & 319.2   \\ 
FPC-AS& 2.08 & 158 & 5.31 & 399.4 & 7.8 & 578.6 & 10.1 & 720.4 & 10.5 & 775   \\ 
SpaRSA& 5.56 & 523.4 & 5.56 & 530 & 6.27 & 588.2 & 7.45 & 671.6 & 8.11 & 759.6   \\ 
\hline
\end{tabular}
\end{table}

\begin{table}[htb]
\caption{Total number of $A$- and $A^T$-  calls $N_A$ and CPU time (in seconds) averaged over 10 independent runs with dynamic range 60 dB}\label{tab:db60}
\setlength{\tabcolsep}{4.7pt}
\centering
\begin{tabular}{ccccccccccc}
\hline
method & \multicolumn{2}{c}{$\epsilon : 10^{0}$} &  \multicolumn{2}{c}{$\epsilon : 10^{-1}$} &  \multicolumn{2}{c}{$\epsilon : 10^{-2}$} & \multicolumn{2}{c}{$\epsilon : 10^{-4}$} &
\multicolumn{2}{c}{$\epsilon : 10^{-6}$}  \\ 
 \cline{2-3}  \cline{4-5}  \cline{6-7}   \cline{8-9}   \cline{10-11}  
& time & $N_A$ & time & $N_A$ & time & $N_A$ &  time & $N_A$ & time & $N_A$ \\ \hline
SNF& 5.66 & 391.8 & 9.31 & 648.8 & 11.1 & 777.6 & 11.8 & 828.2 & 12.5 & 876.6   \\ 
SNF(aCG)& 5.62 & 391.8 & 9.28 & 648.8 & 11 & 777.6 & 12.2 & 861.2 & 12.7 & 889   \\ 
ASSN& 3.92 & 295.4 & 5.38 & 416.4 & 6.45 & 492 & 7.49 & 582.4 & 8.19 & 642.4   \\ 
SSNP& 21.5 & 1607.2 & 29.5 & 2247.6 & 32.2 & 2478.8 & 41.9 & 3236.2 & 50.9 & 3927.4   \\ 
ASLB(2)& 2.11 & 146.2 & 2.89 & 201.6 & 3.54 & 250.6 & 4.5 & 317.6 & 5.42 & 383.4   \\ 
ASLB(1)& 2.11 & 143.8 & 2.66 & 187.8 & 3.25 & 228.2 & 4.22 & 295 & 5.22 & 368.6   \\ 
FPC-AS& 3.02 & 232.2 & 8.84 & 644 & 11.5 & 844.2 & 13.8 & 1004.2 & 14.6 & 1031.8   \\ 
SpaRSA& 6.01 & 561.2 & 6.39 & 598.2 & 7.27 & 683.2 & 8.25 & 797.8 & 9.84 & 900.6   \\ 
\hline
\end{tabular}
\end{table}

\begin{table}[!htb]
\caption{Total number of $A$- and $A^T$-  calls $N_A$ and CPU time (in seconds) averaged over 10 independent runs with dynamic range 80 dB}\label{tab:db80}
\setlength{\tabcolsep}{4.8pt}
\centering
\begin{tabular}{ccccccccccc}
\hline
method & \multicolumn{2}{c}{$\epsilon : 10^{0}$} &  \multicolumn{2}{c}{$\epsilon : 10^{-1}$} &  \multicolumn{2}{c}{$\epsilon : 10^{-2}$} & \multicolumn{2}{c}{$\epsilon : 10^{-4}$} &
\multicolumn{2}{c}{$\epsilon : 10^{-6}$}  \\ 
 \cline{2-3}  \cline{4-5}  \cline{6-7}   \cline{8-9}   \cline{10-11}  
& time & $N_A$ & time & $N_A$ & time & $N_A$ &  time & $N_A$ & time & $N_A$ \\ \hline
SNF& 7.47 & 591 & 10.7 & 841.6 & 12.4 & 978.6 & 13 & 1024.8 & 13.6 & 1057.8   \\ 
SNF(aCG)& 7.56 & 591 & 10.6 & 841.6 & 12.4 & 978.6 & 13.2 & 1042.2 & 13.9 & 1099.4   \\ 
ASSN& 6.39 & 482.8 & 7.66 & 601 & 8.66 & 690.6 & 9.9 & 780.6 & 10.5 & 833.4   \\ 
SSNP& 36.1 & 2820.6 & 34.2 & 2767.2 & 42.7 & 3497 & 51.3 & 4201.4 & 56.6 & 4531.2   \\ 
ASLB(2)& 3.65 & 255.8 & 4.03 & 299.4 & 4.98 & 355.6 & 5.61 & 411.4 & 6.21 & 440   \\ 
ASLB(1)& 3.02 & 213.6 & 3.59 & 258 & 4.24 & 299.2 & 4.95 & 357.6 & 5.52 & 385.6   \\ 
FPC-AS& 4.16 & 321.2 & 8.18 & 611.4 & 10.7 & 788.4 & 12.1 & 886 & 12.1 & 900.8   \\ 
SpaRSA& 5.74 & 543.2 & 6.96 & 665.4 & 8.17 & 763.2 & 9.1 & 873.6 & 9.85 & 930.2   \\ 
\hline
\end{tabular}
\end{table}

We solve the test problems under different tolerances $\epsilon \in \{10^{-0},
10^{-1},10^{-2},10^{-4},10^{-6} \}$ and dynamic ranges $d \in \{20,40,60,80 \}$.
Since the evaluations of $\dct$ dominate the overall computation, we mainly use
the total numbers of $A$- and $A^T$-  calls $N_A$ to compare the efficiency of
different solvers.  Tables \ref{tab:db20} - \ref{tab:db80} show the averaged
numbers of $N_A$ and CPU time over 10 independent trials. These tables show that
ASSN and ASLB are competitive to other methods. For the low accuracy, SpaRSA and
FPC-AS show a fast convergence rate. ASSN and ASLB are both faster than or close
to FPC-AS and SpaRSA regardless of $N_A$ and CPU time in most cases. In the
meanwhile, ASSN and ASLB are competitive to the second-order methods under
moderate accuracy. The CPU time and $N_A$ of ASSN and ASLB are less than the
Newton type solver SNF in almost all cases, especially for the large dynamic
range. ASLB with a memory size $m = 1$ shows the fastest speed in low accuracy.
It is necessary to emphasize that L-BFGS with $m = 1$ is equal to the
Hestenes-Stiefel and Polak-Ribi\`ere conjugate gradient method with exact line
search \cite{NocedalWright06}. Compared with ASSN, SSNP has a slower convergent rate, which implies that our adaptive strategy on switching Newton and projection steps is helpful.

\begin{figure}[htb]
\centering
  \subfigure[20dB]{
    \includegraphics[width=0.45\textwidth,height=0.4\textwidth]
    {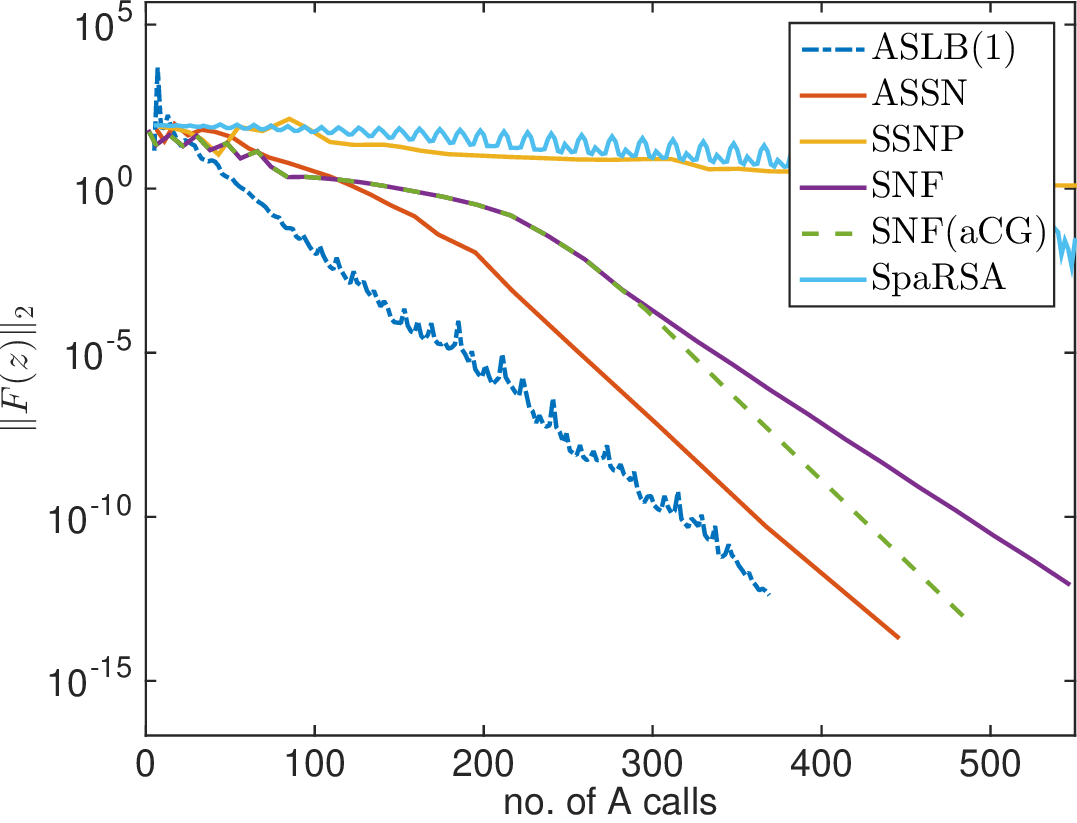}
  }
    \subfigure[40dB]{
    \includegraphics[width=0.45\textwidth,height=0.4\textwidth]
    {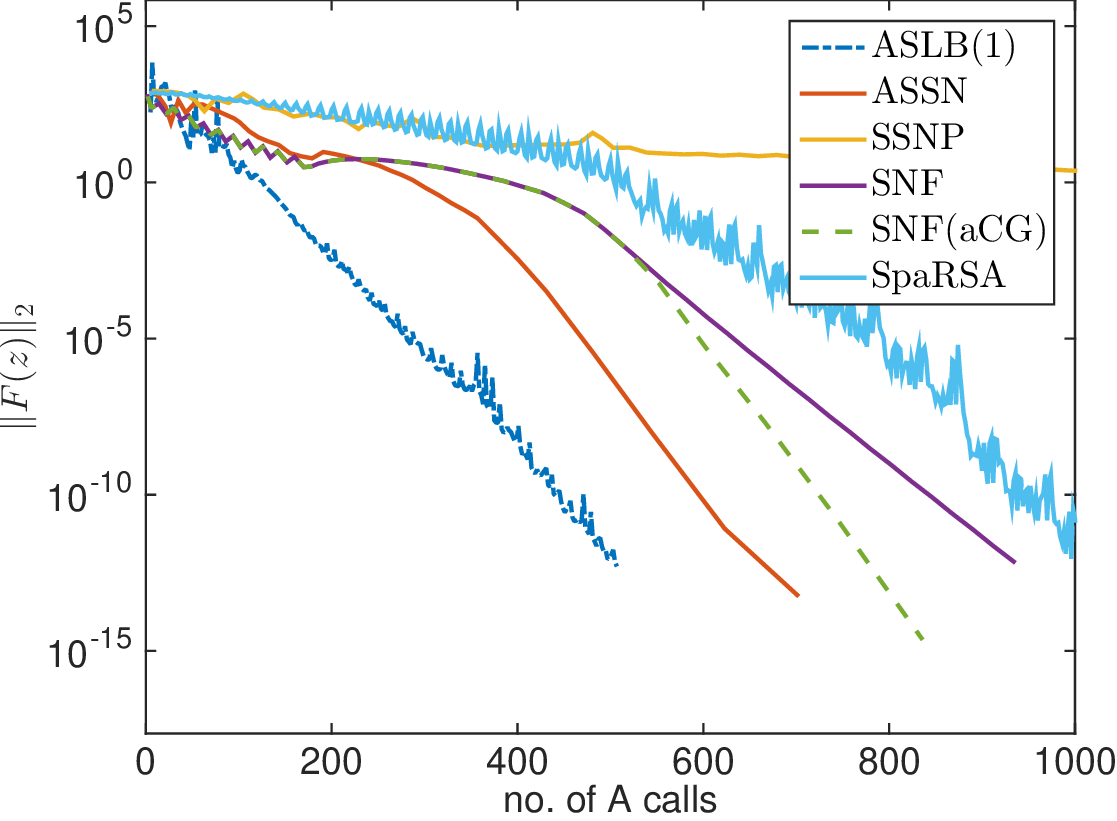}
  }
  \subfigure[60dB]{
    \includegraphics[width=0.45\textwidth,height=0.4\textwidth]
    {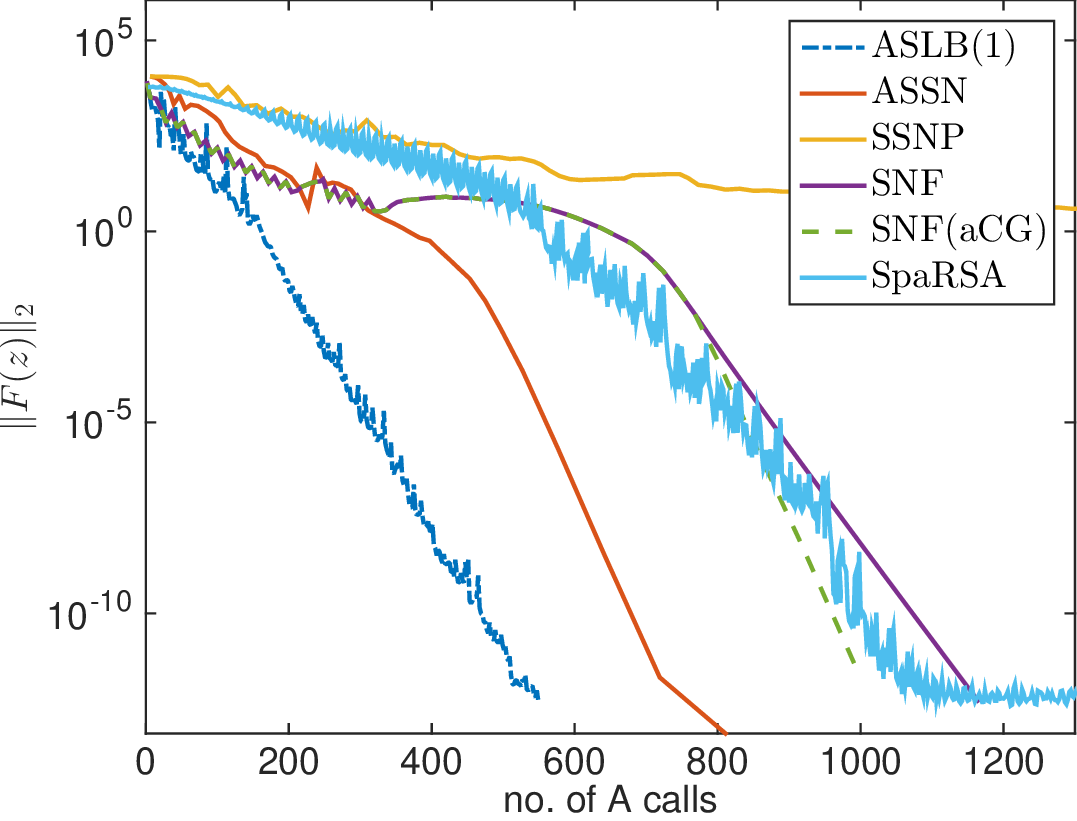}
  }
  \subfigure[80dB]{
    \includegraphics[width=0.45\textwidth,height=0.4\textwidth]
    {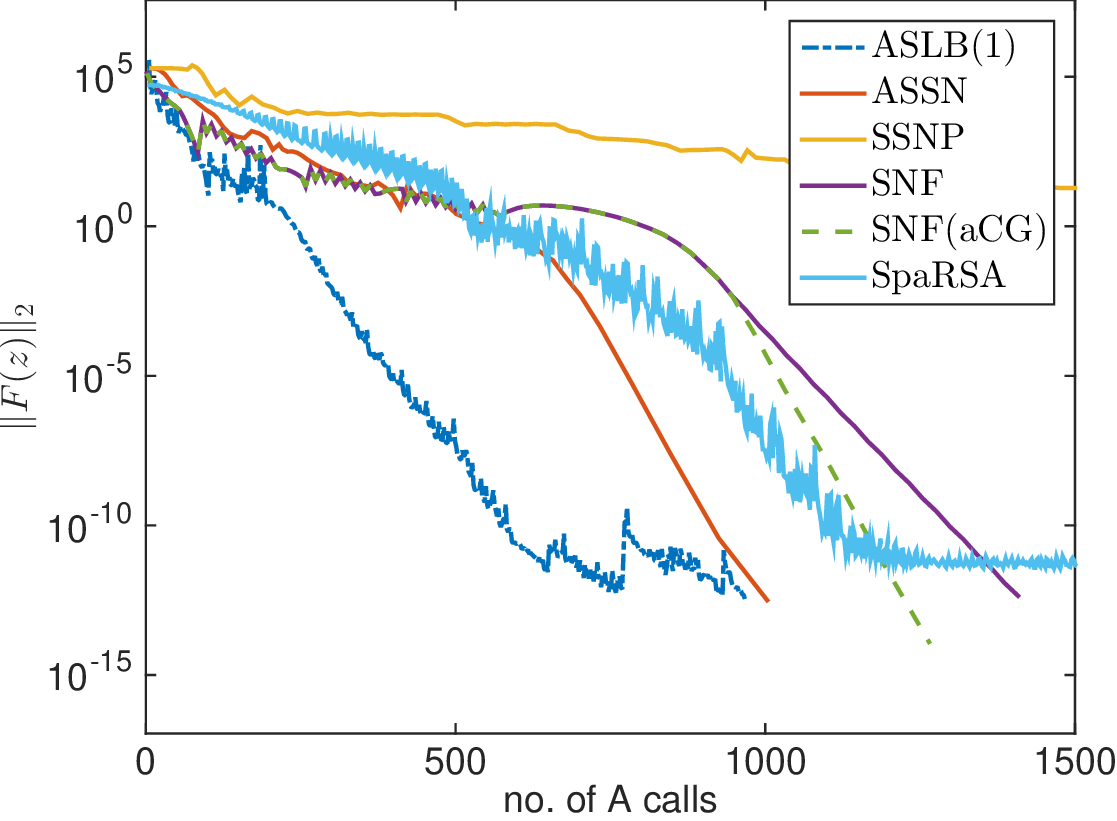}
  }
\caption{residual history  with respect to the total numbers of $A$- and $A^T$-  calls $N_A$}
\label{fig:Acall-res}
\end{figure}

\begin{figure}[htb]
\centering
  \subfigure[20dB]{
    \includegraphics[width=0.45\textwidth,height=0.4\textwidth]
    {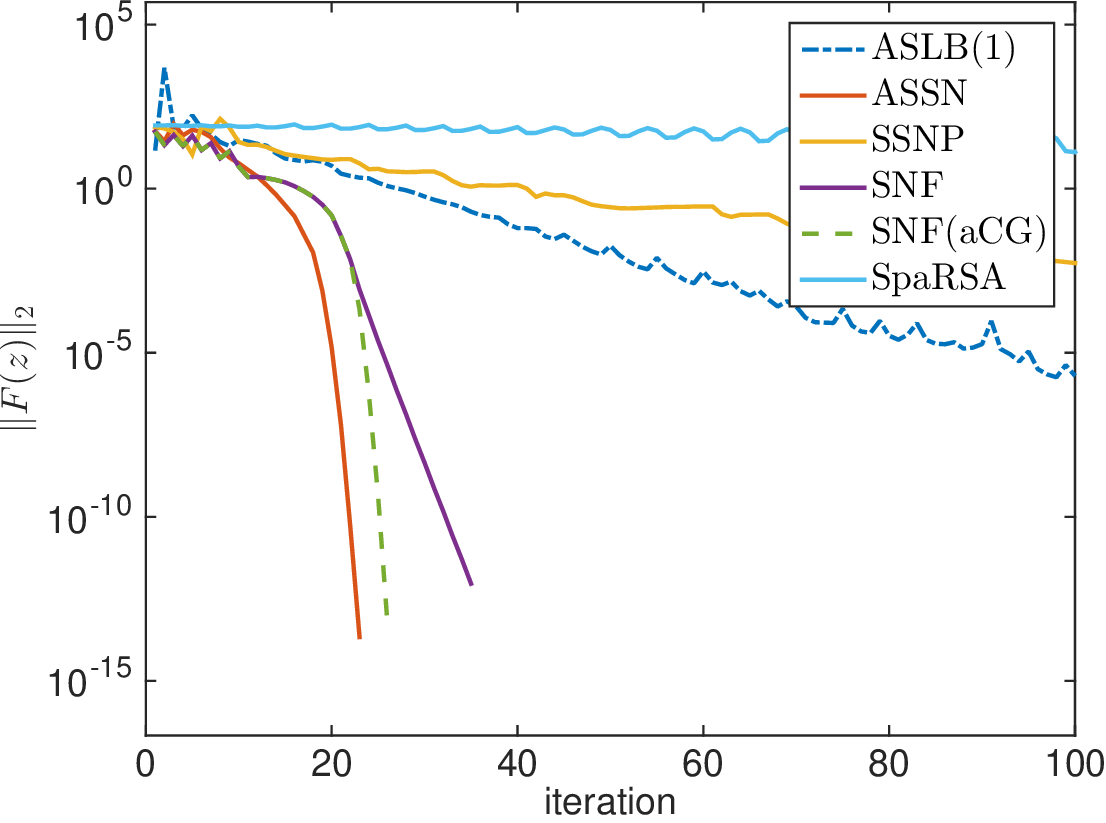}
  }
    \subfigure[40dB]{
    \includegraphics[width=0.45\textwidth,height=0.4\textwidth]
    {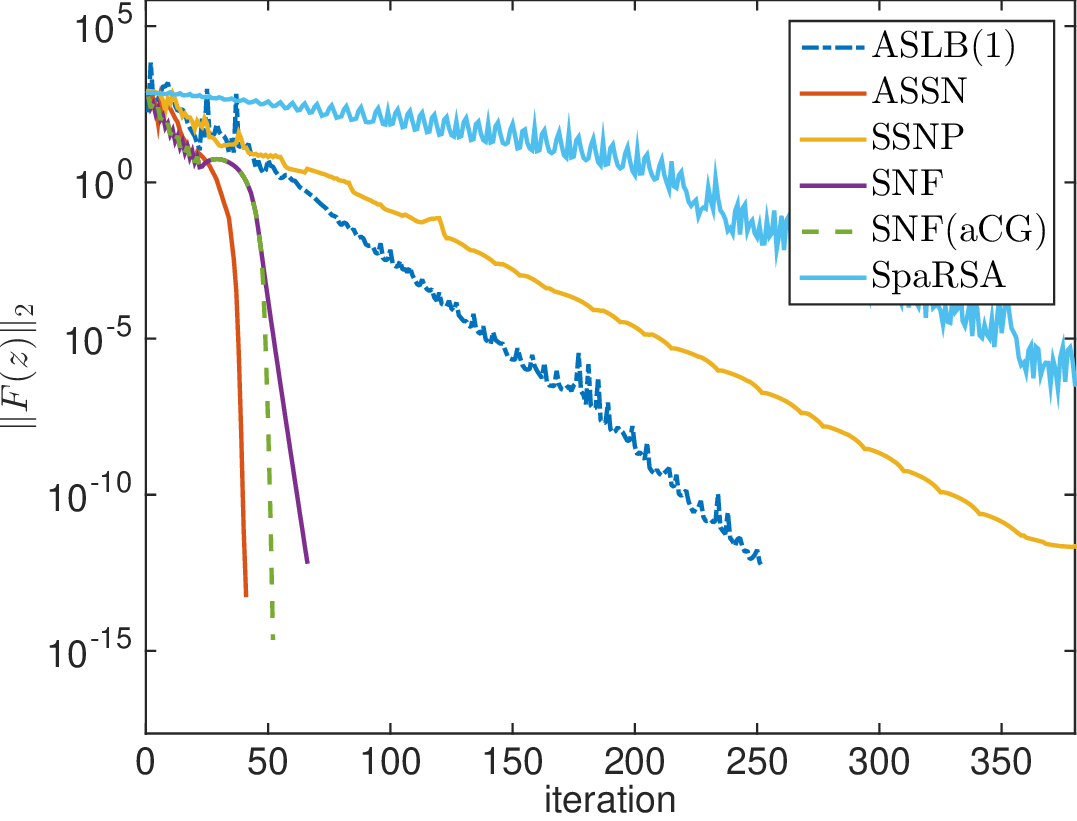}
  }
  \subfigure[60dB]{
    \includegraphics[width=0.45\textwidth,height=0.4\textwidth]
    {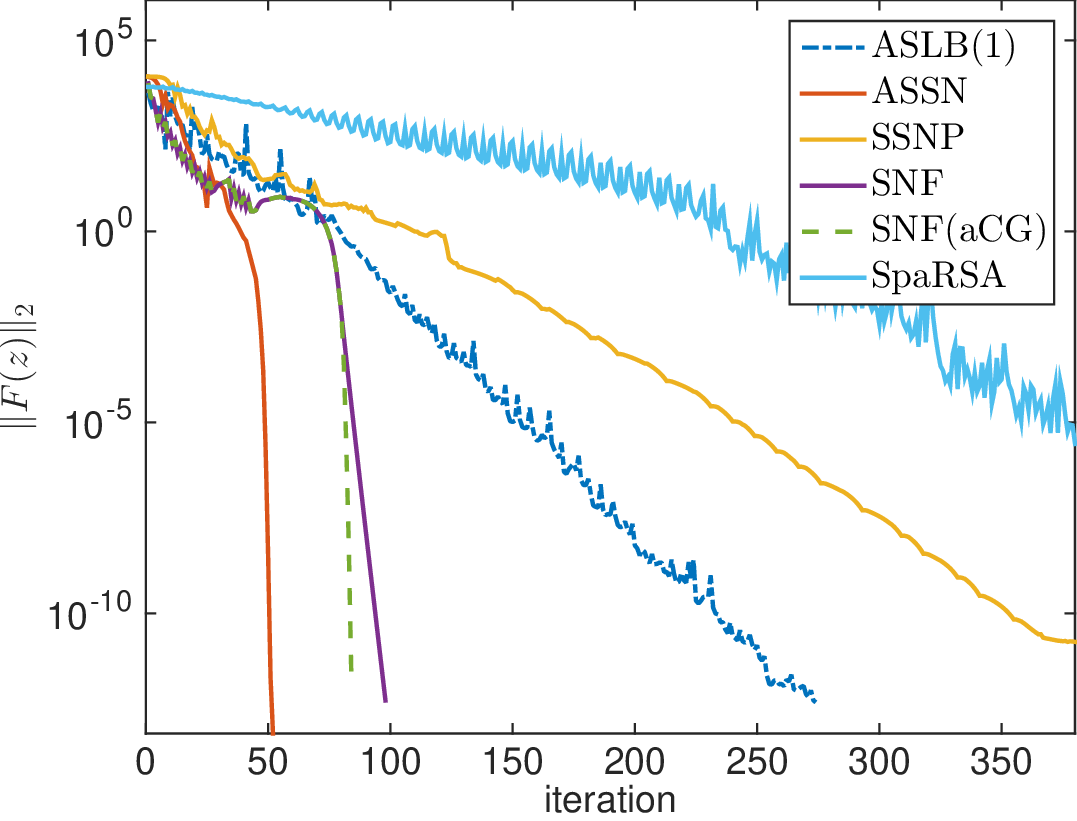}
  }
  \subfigure[80dB]{
    \includegraphics[width=0.45\textwidth,height=0.4\textwidth]
    {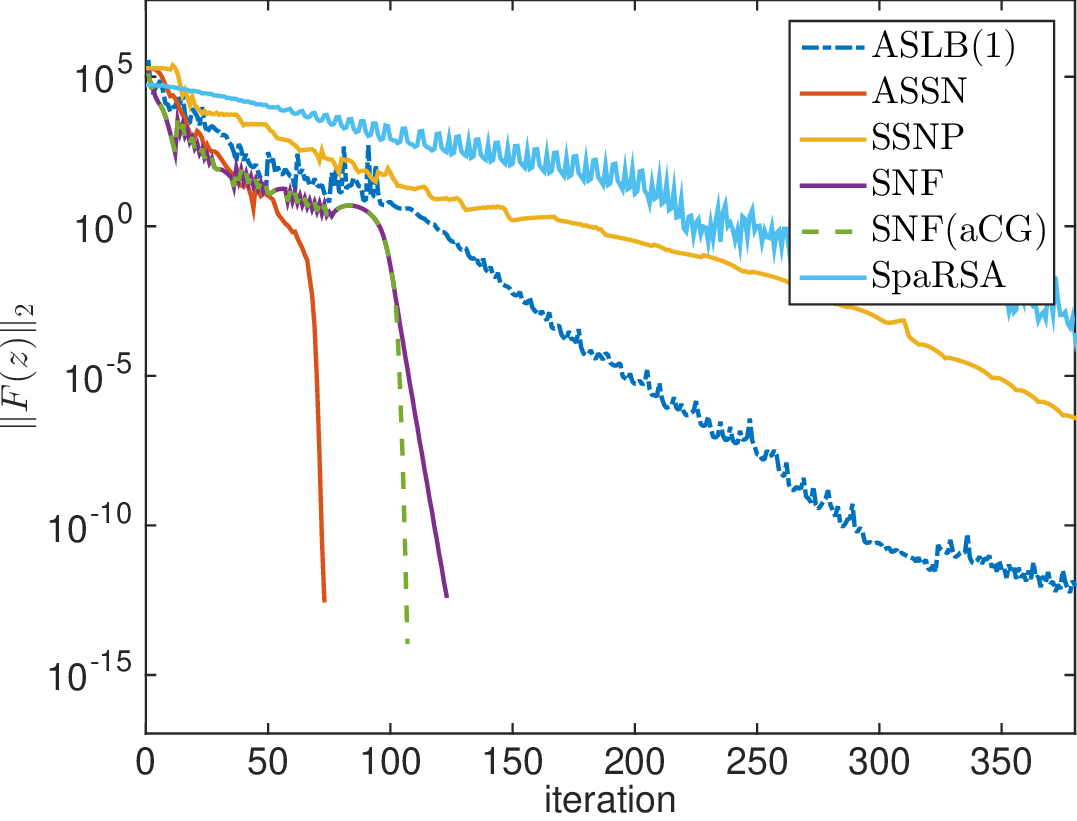}
  }
\caption{residual history  with respect to the total numbers of iteration}
\label{fig:iter-res}
\end{figure}

In particular, ASSN and ASLB have a better performance for high accuracy.
Figures \ref{fig:Acall-res} and \ref{fig:iter-res} illustrate the residual
history with respect to the total number of $A$- and $A^T$-  calls $N_A$ and
the total number of iterations. Since two first-order methods have a close
performance and ASLB(1) performs better than ASLB(2), we omit the the figure of
FPC-AS and ASLB(2). These figures also show  that ASSN and ASLB have a better
performance than SNF and SNF(aCG) independent of dynamic ranges. In particular,
quadratic convergence is observable from ASSN in these examples.

\subsection{Applications to the DRS method} 


Consider the Basis-Pursuit (BP) problem
\be \label{eq:BPDe} \min \; \|x\|_1, \st Ax=b,\ee
where $A\in \R^{m\times n}$ is of full row rank and $b\in \R^m$.
Let $f(x)=\ind_\Omega(Ax-b)$ and $h(x)=\|x\|_1$, where  the
set $\Omega=\{0\}$. 
The system of nonlinear equations corresponding to the
DRS fixed-point mapping is  
\begin{equation}\label{eq:F-BP}
F(z) = \prox_{th}(z)-\prox_{tf}(2\prox_{th}(z)-z)=0.
\end{equation}

For the simplicity of solving the subproblems in the DRS method, we make the
assumption that $AA^\top  =I$.  Then it can be derived that  the proximal mapping with respect
to $f(x)$ is 
\beaa \prox_{tf}(z)&=&  (I-A^\top A) z + A^\top
\left(\prox_{\ind_\Omega}(Az-b)+b\right)\\
&=& 
z- A^\top (A z-b).
\eeaa
A generalized Jacobian matrix $D \in\partial \prox_{tf}((2\prox_{th}(z)-z))$ is taken as follows
\bea 
D
&=&  I- A^\top A.
\eea
The proximal mapping with respect to $h(x)$ is
\[ \left( \prox_{th}(z) \right)_i= \sign(z_i) \max(|z_i|-t, 0).
\]
One can take a generalized Jacobian matrix $M(z)\in \partial \prox_{th}(z)$ as a diagonal matrix with diagonal entries
\[
M_{ii}(z)=\left\{
\begin{array}{ll}
  1,\quad & |(z)_i|>t,\\
0,\quad & \mbox{ otherwise}.\\
\end{array}
\right.
\]
Hence, a generalized Jacobian matrix of $F(z)$
 is  in the form of 
\be \label{eq:G-l1}
J(z)= M(z) + \bA(I-2M(z)).
\ee
Let $W=(I-2M(z))$ and $H=W+M(z)+\mu I$. Using the binomial inverse theorem, we
obtain the inverse matrix
\beaa
(J(z)+\mu I)^{-1} &=& (H - A^\top A W)^{-1} \\
&=&  H^{-1} +  H^{-1} A^\top (I- AW  H^{-1}A^\top)^{-1} A WH^{-1}.
\eeaa
For convenience, we write the  diagonal entries of matrix $W$ and $H$ as
\[
W_{ii}(z)=\left\{
\begin{array}{ll}
-1,\quad & |(z)_i|>t,\\
1,\quad & \mbox{ otherwise}\\
\end{array}
\right.
\text{and} \quad
H_{ii}(z)=\left\{
\begin{array}{ll}
\mu,\quad & |(z)_i|>t,\\
1+\mu,\quad & \mbox{ otherwise}.\\
\end{array}
\right.
\]
Then $WH^{-1} = \frac{1}{1+\mu}I - S$, where $S$ is a diagonal matrix with diagonal entries
\[
S_{ii}(z)=\left\{
\begin{array}{ll}
\frac{1}{\mu} + \frac{1}{1+\mu},\quad & |(z)_i|>t,\\
0,\quad & \mbox{ otherwise}.\\
\end{array}
\right.
\]
Hence, $I- AW  H^{-1}A^\top = (1-\frac{1}{1+\mu})I + ASA^\top$. Define the index sets
\[
\begin{split}
&\mathcal{I}(x) := \{i : |(z)_i| >t  \}  = \{i: M_{ii}(x) = 1\}, \\
&\mathcal{O}(x) := \{i : |(z)_i| \leq t  \}  = \{i: M_{ii}(x) = 0\}
\end{split}
\]
and $A_{\mathcal{I}(x)}$ denote the matrix containing the column $\mathcal{I}(x)$ of A, then we have 
\be\label{eq:spa-drs}
ASA^\top = (\frac{1}{\mu} + \frac{1}{1+\mu})A_{\mathcal{I}(x)}A_{\mathcal{I}(x)}^\top.
\ee
The above property  implies the positive definiteness of $I- AW  H^{-1}A^\top$ and can be used to reduce the computational complexity if the submatrix $A_{\mathcal{I}(x)} A_{\mathcal{I}(x)} ^\top$ is easily available.


\subsubsection{Numerical comparison} 
In this subsection, we compare our methods with two first-order solvers: ADMM
\cite{YF2001} and SPGL1 \cite{vdBF2008}.
 The ASLB solver is not included since its performance is not comparable with other approaches. 
 Our test problems are almost the same as
the last subsection and the only difference is that we set $b = A\bar{x}$ without
adding noise. We use the residual criterion  $\|F(z)\|\le \epsilon$ as the
stopping criterion for ADMM and ASSN. Because the computation of residual of
SPGL1 needs extra cost, we use its original criterion and list the relative
error ``rerr" to compare with ADMM and ASSN. The relative error with respect to the true solution $x^*$ is denoted by  
\[
\mbox{rerr} = \frac{||x^k - x^*||}{\max(||x^*||,1)}.
\]
We revise the ADMM in yall1\footnote{downloadable from
\url{http://yall1.blogs.rice.edu}} by adjusting the rules of updating the
penalty parameter and choosing the best parameters so that it can solve all
examples in our numerical experiments. The parameters are set to the default
values in SPGL1. Since the matrix $A$ is only available as an operator, the property \eqref{eq:spa-drs} cannot be applied in ASSN. 
\begin{table}[htb]
\caption{Total number of $A$- and $A^T$-  calls $N_A$, CPU time (in seconds) and relative error with dynamic range 20 dB}\label{tab:drs20}
\setlength{\tabcolsep}{4.8pt}
\begin{tabular}{cccccccccc}
\hline
method & \multicolumn{3}{c}{$\epsilon : 10^{-2}$} &  \multicolumn{3}{c}{$\epsilon : 10^{-4}$} &  \multicolumn{3}{c}{$\epsilon : 10^{-6}$}  \\ 
\cline{2-4}  \cline{5-7}  \cline{8-10}   
& time & $N_A$ & rerr & time & $N_A$ & rerr & time & $N_A$ & rerr  \\ \hline
ADMM& 10.9 & 646 & 2.781e-04 & 14 & 1026 & 2.658e-06 & 19.4 & 1438 & 2.467e-08 \\ 
ASSN& 8.58 & 694 & 1.175e-04 & 9.73 & 734 & 2.811e-06 & 10.7 & 813 & 4.282e-09 \\ 
SPGL1& 17.3 & 733 & 2.127e-01 & 54.4 & 2343 & 2.125e-01 & 72.3 & 3232 & 2.125e-01 \\ 
\hline
\end{tabular}
\end{table}

\begin{table}[htb]
\caption{Total number of $A$- and $A^T$-  calls $N_A$, CPU time (in seconds) and relative error with dynamic range 40 dB}\label{tab:drs40}
\setlength{\tabcolsep}{4.8pt}
\begin{tabular}{cccccccccc}
\hline
method & \multicolumn{3}{c}{$\epsilon : 10^{-2}$} &  \multicolumn{3}{c}{$\epsilon : 10^{-4}$} &  \multicolumn{3}{c}{$\epsilon : 10^{-6}$}  \\ 
\cline{2-4}  \cline{5-7} \cline{8-10}  
& time & $N_A$ & rerr & time & $N_A$ & rerr & time & $N_A$ & rerr  \\ \hline
ADMM& 6.92 & 504 & 2.092e-04 & 12 & 875 & 2.623e-06 & 17.3 & 1306 & 2.926e-08 \\ 
ASSN& 5.79 & 469 & 7.595e-05 & 7.19 & 582 & 8.922e-07 & 8.43 & 632 & 2.006e-08 \\ 
SPGL1& 29.8 & 1282 & 2.350e-02 & 58.5 & 2477 & 2.346e-02 & 68.1 & 2910 & 2.346e-02 \\ 
\hline
\end{tabular}
\end{table}

\begin{table}[!htb]
\caption{Total number of $A$- and $A^T$-  calls $N_A$, CPU time (in seconds) and relative error with dynamic range 60 dB}\label{tab:drs60}
\setlength{\tabcolsep}{4.8pt}
\begin{tabular}{cccccccccc}
\hline
method & \multicolumn{3}{c}{$\epsilon : 10^{-2}$} &  \multicolumn{3}{c}{$\epsilon : 10^{-4}$} &  \multicolumn{3}{c}{$\epsilon : 10^{-6}$}  \\ 
\cline{2-4}  \cline{5-7} \cline{8-10}  
& time & $N_A$ & rerr & time & $N_A$ & rerr & time & $N_A$ & rerr  \\ \hline
ADMM& 7.44 & 599 & 1.901e-03 & 13.5 & 980 & 2.501e-06 & 18.7 & 1403 & 2.913e-08 \\ 
ASSN& 5.48 & 449 & 1.317e-03 & 9.17 & 740 & 1.922e-06 & 10.2 & 802 & 1.930e-08 \\ 
SPGL1& 55.3 & 2367 & 5.020e-03 & 70.7 & 2978 & 5.017e-03 & 89.4 & 3711 & 5.017e-03 \\ 
\hline
\end{tabular}
\end{table}

\begin{table}[!htb]
\caption{Total number of $A$- and $A^T$-  calls $N_A$, CPU time (in seconds) and relative error with dynamic range 80 dB}\label{tab:drs80}
\setlength{\tabcolsep}{4.8pt}
\begin{tabular}{cccccccccc}
\hline
method & \multicolumn{3}{c}{$\epsilon : 10^{-2}$} &  \multicolumn{3}{c}{$\epsilon : 10^{-4}$} &  \multicolumn{3}{c}{$\epsilon : 10^{-6}$}  \\ 
\cline{2-4}  \cline{5-7} \cline{8-10}  
& time & $N_A$ & rerr & time & $N_A$ & rerr & time & $N_A$ & rerr  \\ \hline
ADMM& 7.8 & 592 & 5.384e-04 & 13.8 & 1040 & 2.481e-06 & 17.7 & 1405 & 2.350e-08 \\ 
ASSN& 4.15 & 344 & 5.194e-04 & 7.92 & 618 & 1.205e-06 & 8.74 & 702 & 5.616e-09 \\ 
SPGL1& 32.2 & 1368 & 4.862e-04 & 56.1 & 2396 & 4.859e-04 & 67.4 & 2840 & 4.859e-04 \\ 
\hline
\end{tabular}
\end{table}

\begin{figure}[htb]
\centering
  \subfigure[20dB]{
    \includegraphics[width=0.45\textwidth,height=0.4\textwidth]
    {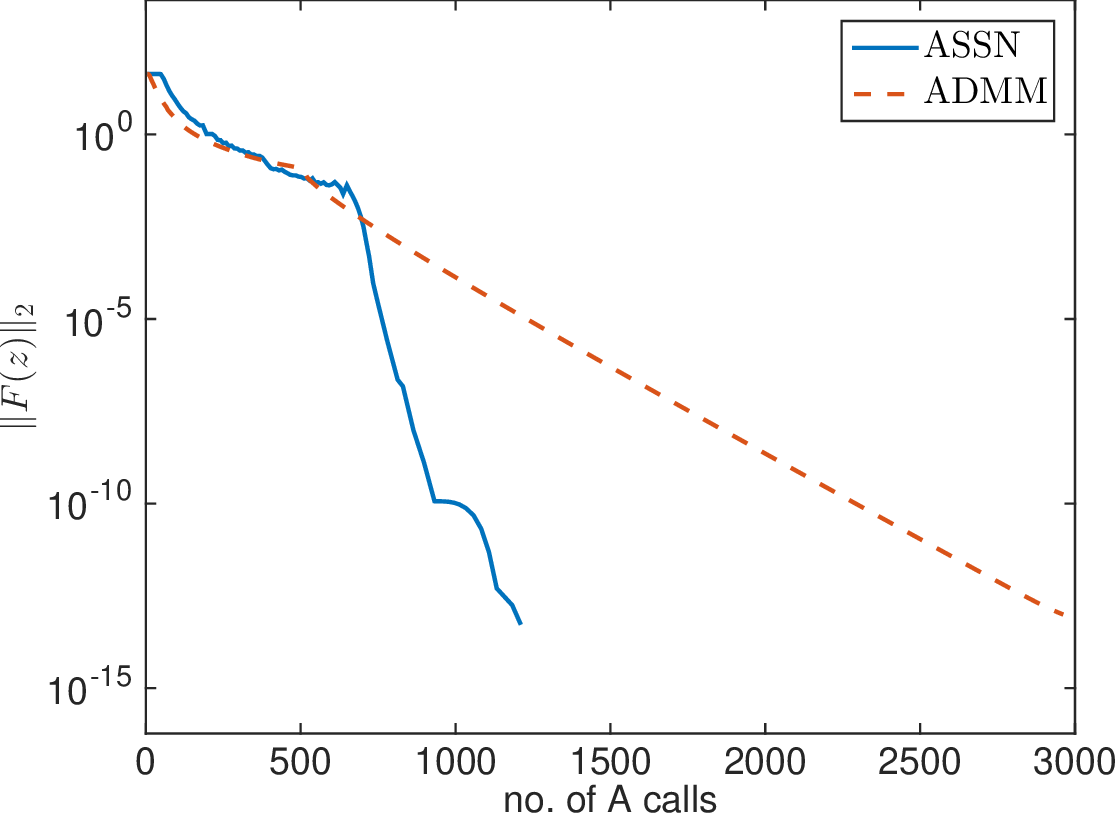}
  }
    \subfigure[40dB]{
    \includegraphics[width=0.45\textwidth,height=0.4\textwidth]
    {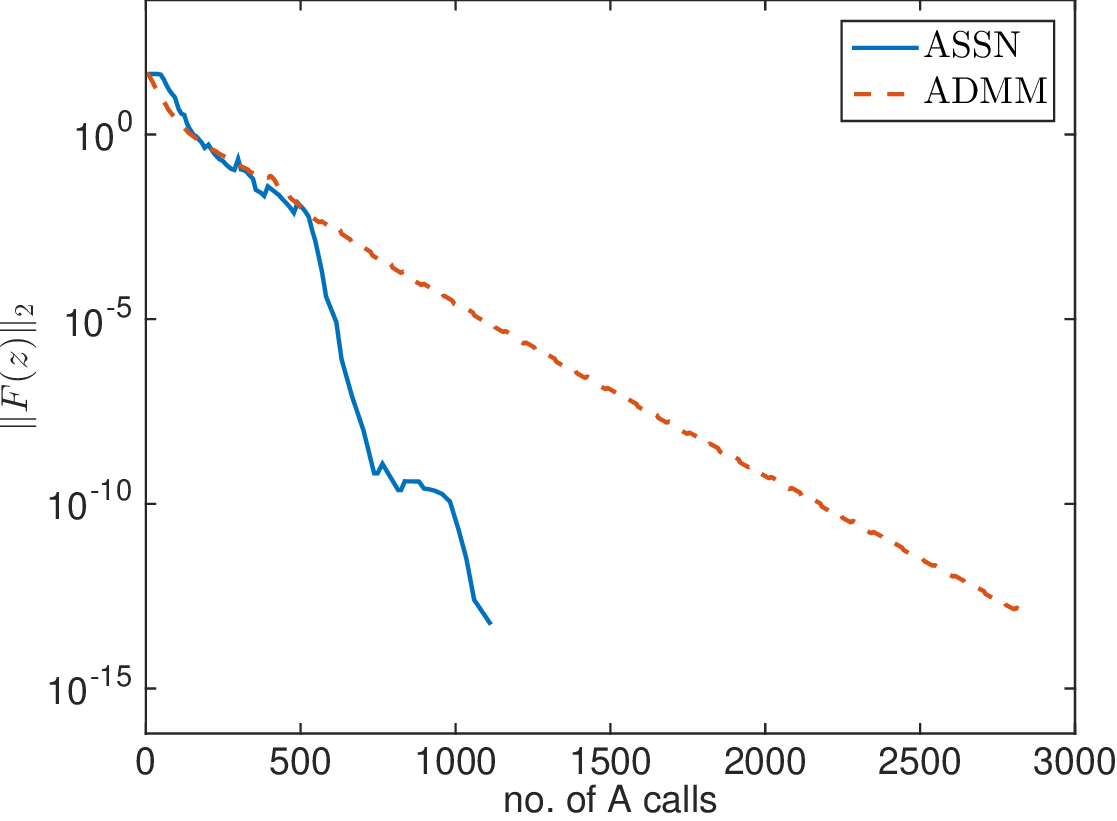}
  }
  \subfigure[60dB]{
    \includegraphics[width=0.45\textwidth,height=0.4\textwidth]
    {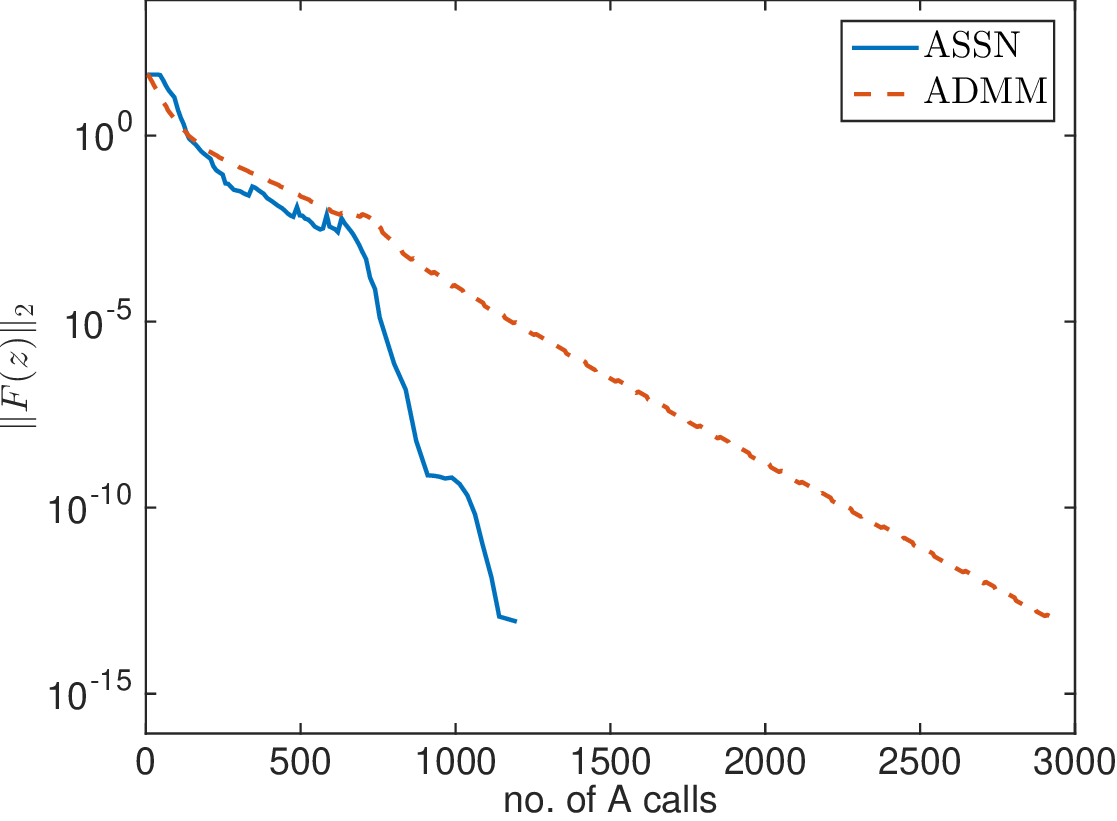}
  }
  \subfigure[80dB]{
    \includegraphics[width=0.45\textwidth,height=0.4\textwidth]
    {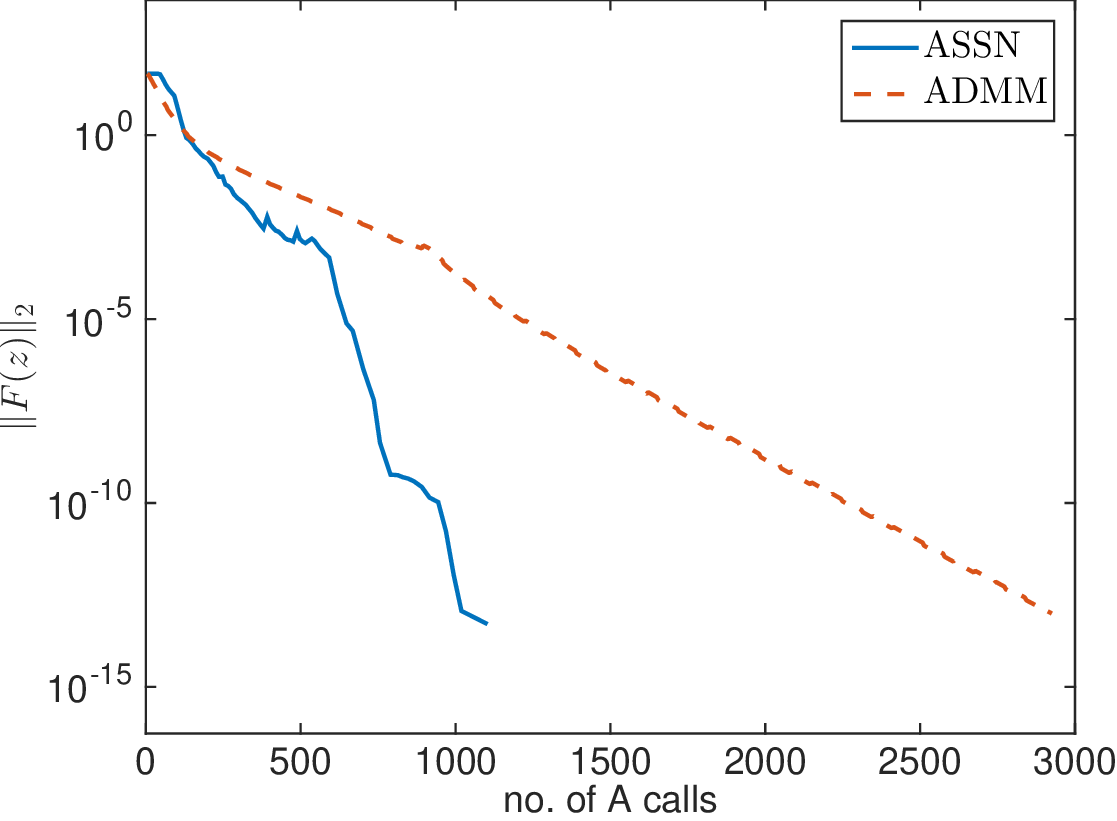}
  }
\caption{residual history  with respect to the total numbers of $A$- and $A^T$-  calls $N_A$}
\label{fig:Acall-res-drs}
\end{figure}

\begin{figure}[htb]
\centering
  \subfigure[20dB]{
    \includegraphics[width=0.45\textwidth,height=0.4\textwidth]
    {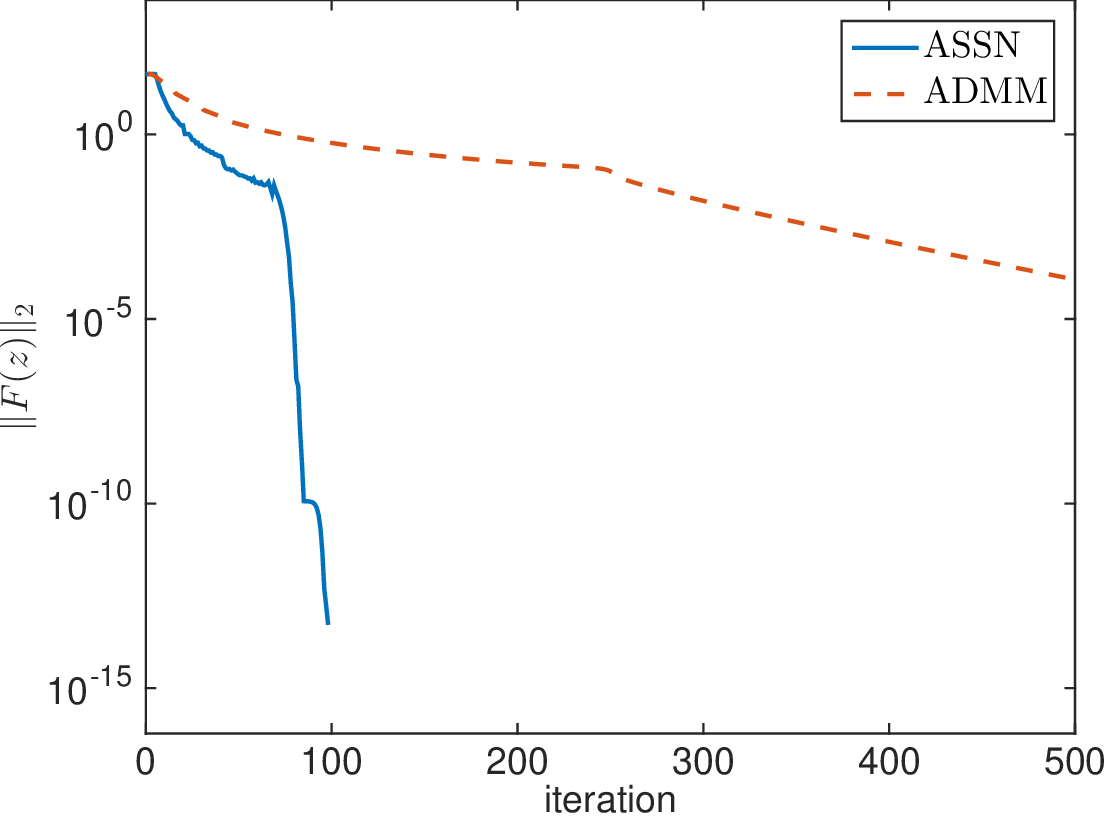}
  }
    \subfigure[40dB]{
    \includegraphics[width=0.45\textwidth,height=0.4\textwidth]
    {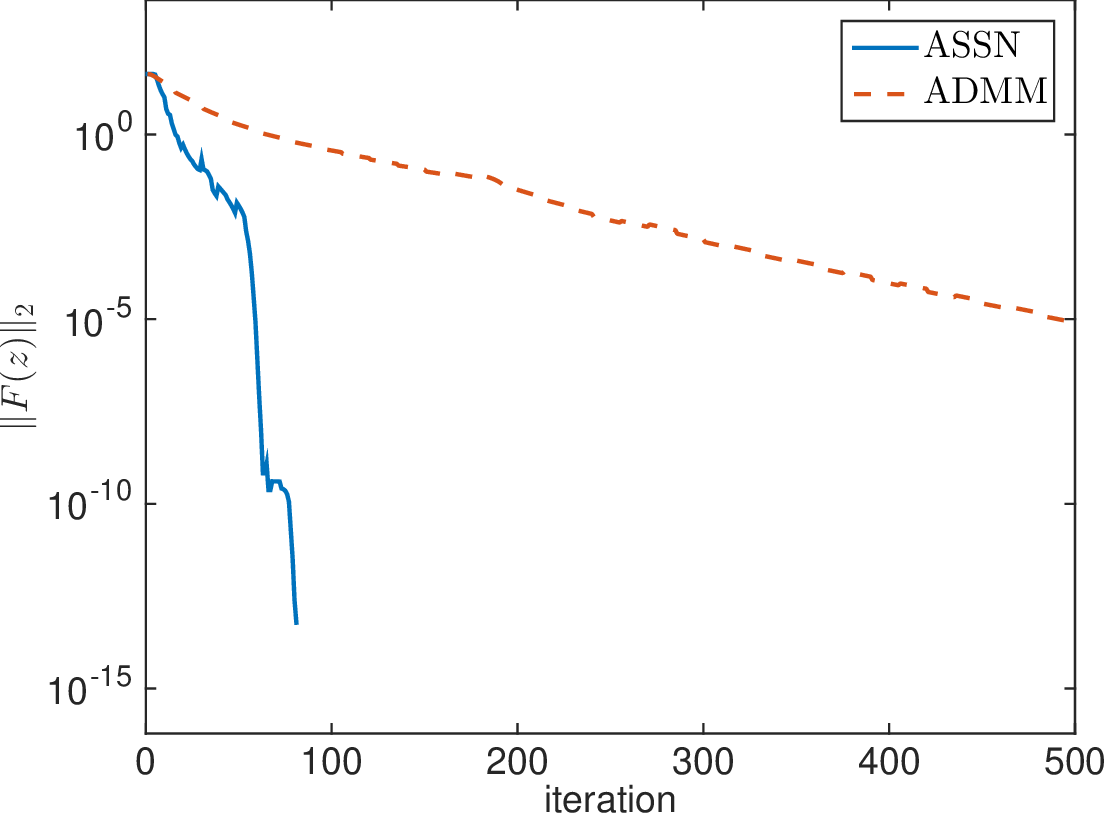}
  }
  \subfigure[60dB]{
    \includegraphics[width=0.45\textwidth,height=0.4\textwidth]
    {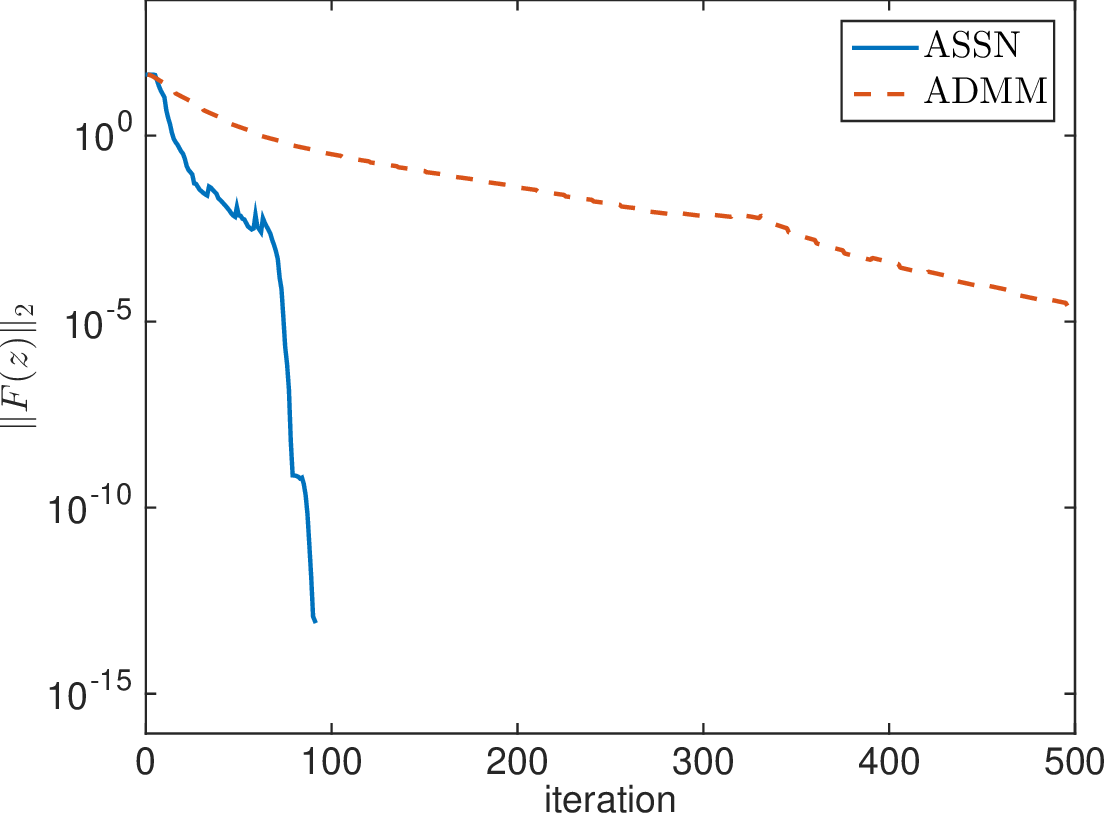}
  }
  \subfigure[80dB]{
    \includegraphics[width=0.45\textwidth,height=0.4\textwidth]
    {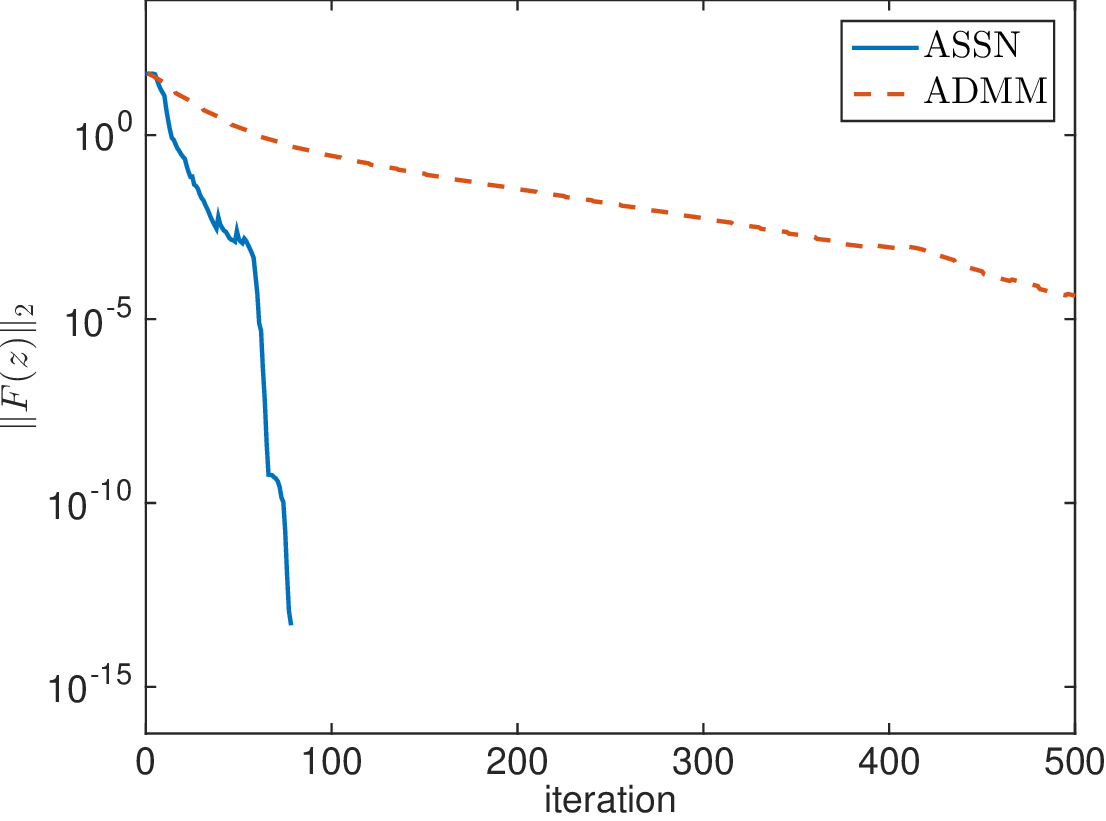}
  }
\caption{residual history with respect to the total numbers of iteration}
\label{fig:iter-res-drs}
\end{figure}

We solve the test problems under different tolerances $\epsilon \in
\{10^{-2},10^{-4},10^{-6} \}$ and dynamic ranges $d \in \{20,40,60,80 \}$.
Similar to the last subsection, we mainly use the total numbers of $A$- and $A^T$-  calls $N_A$ and CPU time to compare the efficiency among different solvers. We also list the relative error so that we can compare ADMM, ASSN with SPGl1. These numerical results are reported in Tables \ref{tab:drs20} - \ref{tab:drs80}. 
The performance of ASSN is close to ADMM for tolerance $10^{-2}$ and is much
better for tolerance $10^{-4}$ and $10^{-6}$ independent of dynamic ranges. For
all test problems, SPGL1 can only obtain a low accurate solution. It may be
improved if the parameters are further tuned. 

Figures \ref{fig:Acall-res-drs} and \ref{fig:iter-res-drs} illustrate the
residual history with respect to the total number of $A$- and $A^T$-  calls
$N_A$ and the total number of iterations. SPGL1 is omitted since it cannot
converge for a high accuracy. The figures show that ASSN has a similar
convergent rate as ADMM in the initial stage but it achieves a faster convergent
rate later, in particular, for a high accuracy. 

\section{Conclusion}\label{sec:conclusion}
 The purpose of this paper is to study second-order type methods  for solving
 composite convex programs based on fixed-point mappings induced from many
 operator splitting approaches such as the FBS and DRS methods. The
 semi-smooth  Newton method is theoretically guaranteed to converge to a global
 solution from an arbitrary initial point and achieve a fast convergent rate by using an adapt strategy on switching the projection steps and Newton steps. 
 Our proposed algorithms are suitable to constrained convex programs when a fixed-point mapping is well-defined. It may be able to bridge the gap between first-order and second-order type methods. They are indeed promising from our preliminary numerical experiments on a number of applications. In particular, quadratic or superlinear convergence is attainable in some examples of  Lasso regression and basis pursuit. 

There are a number of future directions worth pursuing from this point on, including theoretical analysis and a comprehensive implementation of these second-order algorithms. To improve the performance in practice, the second-order methods can be activated until the first-order type methods reach a good neighborhood of the global optimal solution. Since solving the corresponding system of linear equations is computationally dominant, it is important to explore the structure of the linear system and design certain suitable preconditioners. 
  
\section*{Acknowledgements} 
The authors would like to thank Professor Defeng Sun for the valuable
discussions on semi-smooth Newton methods, and Professor Michael Ulbrich and Dr.  Andre Milzarek for sharing their code SNF. 

\bibliographystyle{siam}
\bibliography{ref.bib}

\begin{thebibliography}{10}

\bibitem{AAB2013}
{\sc M.~Ahookhosh, K.~Amini, and S.~Bahrami}, {\em Two derivative-free
  projection approaches for systems of large-scale nonlinear monotone
  equations}, Numer. Algorithms, 64 (2013), pp.~21--42.

\bibitem{BC2011}
{\sc H.~H. Bauschke and P.~L. Combettes}, {\em Convex analysis and monotone
  operator theory in {H}ilbert spaces}, Springer, New York, 2011.

\bibitem{BPCPE2011}
{\sc S.~Boyd, N.~Parikh, E.~Chu, B.~Peleato, and J.~Eckstein}, {\em Distributed
  optimization and statistical learning via the alternating direction method of
  multipliers}, Foundations and Trends in Machine Learning, 3 (2011),
  pp.~1--122.

\bibitem{BCNO2015}
{\sc R.~H. Byrd, G.~M. Chin, J.~Nocedal, and F.~Oztoprak}, {\em A family of
  second-order methods for convex $\ell_1$-regularized optimization}.
\newblock Math. Program., online first, DOI:10.1007/s10107-015-0965-3, 2015.

\bibitem{CLST2016}
{\sc C.~Chen, Y.~J. Liu, D.~Sun, and K.~C. Toh}, {\em A semismooth
  {N}ewton-{CG} based dual {PPA} for matrix spectral norm approximation
  problems}, Math. Program., 155 (2016), pp.~435--470.

\bibitem{CP2011}
{\sc P.~L. Combettes and J.-C. Pesquet}, {\em Proximal splitting methods in
  signal processing}, in Fixed-point algorithms for inverse problems in science
  and engineering, vol.~49 of Springer Optim. Appl., Springer, New York, 2011,
  pp.~185--212.

\bibitem{DY2014I}
{\sc D.~Davis and W.~Yin}, {\em Convergence rate analysis of several splitting
  schemes}.
\newblock http://arxiv.org/abs/1406.4834v3, 05 2015.

\bibitem{DR1956}
{\sc J.~Douglas and H.~H. Rachford}, {\em On the numerical solution of heat
  conduction problems in two and three space variables}, Trans. Amer. Math.
  Soc., 82 (1956), pp.~421--439.

\bibitem{DL2016}
{\sc D.~Drusvyatskiy and A.~S. Lewis}, {\em Error bounds, quadratic growth, and
  linear convergence of proximal methods}.
\newblock http://arxiv.org/abs/1602.06661, 02 2016.

\bibitem{EB1992}
{\sc J.~Eckstein and D.~P. Bertsekas}, {\em On the {D}ouglas-{R}achford
  splitting method and the proximal point algorithm for maximal monotone
  operators}, Math. Program., 55 (1992), pp.~293--318.

\bibitem{FP2003I}
{\sc F.~Facchinei and J.-S. Pang}, {\em Finite-dimensional variational
  inequalities and complementarity problems. {V}ol. {I}}, Springer-Verlag, New
  York, 2003.

\bibitem{FP2003II}
\leavevmode\vrule height 2pt depth -1.6pt width 23pt, {\em Finite-dimensional
  variational inequalities and complementarity problems. {V}ol. {II}},
  Springer-Verlag, New York, 2003.

\bibitem{FY2005}
{\sc J.~Y. Fan and Y.~X. Yuan}, {\em On the quadratic convergence of the
  {L}evenberg-{M}arquardt method without nonsingularity assumption}, Computing,
  74 (2005), pp.~23--39.

\bibitem{GM1976}
{\sc D.~Gabay and B.~Mercier}, {\em A dual algorithm for the solution of
  nonlinear variational problems via finite element approximation}, Computers
  and Mathematics with Applications, 2 (1976), pp.~17--40.

\bibitem{GM1975}
{\sc R.~Glowinski and A.~Marrocco}, {\em Sur l'approximation, par \'el\'ements
  finis d'ordre un, et la r\'esolution, par p\'enalisation-dualit\'e, d'une
  classe de probl\`emes de {D}irichlet non lin\'eaires}, Rev. Fran\c caise
  Automat. Informat. Recherche Op\'erationnelle S\'er. Rouge Anal. Num\'er., 9
  (1975), pp.~41--76.

\bibitem{GL2008}
{\sc R.~Griesse and D.~A. Lorenz}, {\em A semismooth {N}ewton method for
  {T}ikhonov functionals with sparsity constraints}, Inverse Problems, 24
  (2008), pp.~035007, 19.

\bibitem{HYZ2008}
{\sc E.~T. Hale, W.~Yin, and Y.~Zhang}, {\em Fixed-point continuation for
  {$l\sb 1$}-minimization: methodology and convergence}, SIAM J. Optim., 19
  (2008), pp.~1107--1130.

\bibitem{HSZ2015}
{\sc D.~Han, D.~Sun, and L.~Zhang}, {\em Linear rate convergence of the
  alternating direction method of multipliers for convex composite quadratic
  and semi-definite programming}.
\newblock http://arxiv.org/abs/1508.02134, 8 2015.

\bibitem{LSS2014}
{\sc J.~D. Lee, Y.~Sun, and M.~A. Saunders}, {\em Proximal {N}ewton-type
  methods for minimizing composite functions}, SIAM J. Optim., 24 (2014),
  pp.~1420--1443.

\bibitem{LL2011}
{\sc Q.~Li and D.~H. Li}, {\em A class of derivative-free methods for
  large-scale nonlinear monotone equations}, IMA J. Numer. Anal., 31 (2011),
  pp.~1625--1635.

\bibitem{LM1979}
{\sc P.-L. Lions and B.~Mercier}, {\em Splitting algorithms for the sum of two
  nonlinear operators}, SIAM J. Numer. Anal., 16 (1979), pp.~964--979.

\bibitem{LT1993}
{\sc Z.~Q. Luo and P.~Tseng}, {\em Error bounds and convergence analysis of
  feasible descent methods: a general approach}, Ann. Oper. Res., 46 (1993),
  pp.~157--178.

\bibitem{Mifflin1977}
{\sc R.~Mifflin}, {\em Semismooth and semiconvex functions in constrained
  optimization}, SIAM J. Control Optim., 15 (1977), pp.~959--972.

\bibitem{MU2014}
{\sc A.~Milzarek and M.~Ulbrich}, {\em A semismooth {N}ewton method with
  multidimensional filter globalization for {$l_1$}-optimization}, SIAM J.
  Optim., 24 (2014), pp.~298--333.

\bibitem{NocedalWright06}
{\sc J.~Nocedal and S.~J. Wright}, {\em Numerical Optimization}, Springer
  Series in Operations Research and Financial Engineering, Springer, New York,
  second~ed., 2006.

\bibitem{OR1970}
{\sc J.~M. Ortega and W.~C. Rheinboldt}, {\em Iterative solution of nonlinear
  equations in several variables}, Academic Press, New York-London, 1970.

\bibitem{PQ1993}
{\sc J.-S. Pang and L.~Q. Qi}, {\em Nonsmooth equations: motivation and
  algorithms}, SIAM J. Optim., 3 (1993), pp.~443--465.

\bibitem{PSB2014}
{\sc P.~Patrinos, L.~Stella, and A.~Bemporad}, {\em Forward-backward truncated
  {Newton} methods for convex composite optimization}.
\newblock http://arxiv.org/abs/1402.6655, 02 2014.

\bibitem{Qi1993}
{\sc L.~Q. Qi}, {\em Convergence analysis of some algorithms for solving
  nonsmooth equations}, Math. Oper. Res., 18 (1993), pp.~227--244.

\bibitem{QS1993}
{\sc L.~Q. Qi and J.~Sun}, {\em A nonsmooth version of {N}ewton's method},
  Math. Programming, 58 (1993), pp.~353--367.

\bibitem{RW1998}
{\sc R.~T. Rockafellar and R.~J.-B. Wets}, {\em Variational analysis},
  Springer-Verlag, Berlin, 1998.

\bibitem{Scholtes2012}
{\sc S.~Scholtes}, {\em Introduction to piecewise differentiable equations},
  Springer Briefs in Optimization, Springer, New York, 2012.

\bibitem{Shapiro1994}
{\sc A.~Shapiro}, {\em Directionally nondifferentiable metric projection}, J.
  Optim. Theory Appl., 81 (1994), pp.~203--204.

\bibitem{SS1999}
{\sc M.~V. Solodov and B.~F. Svaiter}, {\em A globally convergent inexact
  {N}ewton method for systems of monotone equations}, in Reformulation:
  nonsmooth, piecewise smooth, semismooth and smoothing methods ({L}ausanne,
  1997), M.~Fukushima and L.~Qi, eds., vol.~22, Kluwer Academic Publishers,
  Dordrecht, 1999, pp.~355--369.

\bibitem{SS2002}
{\sc D.~Sun and J.~Sun}, {\em Semismooth matrix-valued functions}, Math. Oper.
  Res., 27 (2002), pp.~150--169.

\bibitem{Tseng2010}
{\sc P.~Tseng}, {\em Approximation accuracy, gradient methods, and error bound
  for structured convex optimization}, Math. Program., 125 (2010),
  pp.~263--295.

\bibitem{Ulbrich2011}
{\sc M.~Ulbrich}, {\em Semismooth {N}ewton methods for variational inequalities
  and constrained optimization problems in function spaces}, Society for
  Industrial and Applied Mathematics (SIAM), Philadelphia, PA; Mathematical
  Optimization Society, Philadelphia, PA, 2011.

\bibitem{vdBF2008}
{\sc E.~van~den Berg and M.~P. Friedlander}, {\em Probing the {P}areto frontier
  for basis pursuit solutions}, SIAM J. Sci. Comput., 31 (2008), pp.~890--912.

\bibitem{WYGZ2010}
{\sc Z.~Wen, W.~Yin, D.~Goldfarb, and Y.~Zhang}, {\em A fast algorithm for
  sparse reconstruction based on shrinkage, subspace optimization, and
  continuation}, SIAM J. Sci. Comput., 32 (2010), pp.~1832--1857.

\bibitem{WNF2009}
{\sc S.~J. Wright, R.~D. Nowak, and M.~A.~T. Figueiredo}, {\em Sparse
  reconstruction by separable approximation}, IEEE Trans. Signal Process., 57
  (2009), pp.~2479--2493.

\bibitem{YF2001}
{\sc N.~Yamashita and M.~Fukushima}, {\em On the rate of convergence of the
  {L}evenberg-{M}arquardt method}, in Topics in numerical analysis, vol.~15 of
  Comput. Suppl., Springer, Vienna, 2001, pp.~239--249.

\bibitem{ZST2010}
{\sc X.~Y. Zhao, D.~Sun, and K.~C. Toh}, {\em A {N}ewton-{CG} augmented
  {L}agrangian method for semidefinite programming}, SIAM J. Optim., 20 (2010),
  pp.~1737--1765.

\bibitem{ZL2001}
{\sc Y.-B. Zhao and D.~Li}, {\em Monotonicity of fixed point and normal
  mappings associated with variational inequality and its application}, SIAM J.
  Optim., 11 (2001), pp.~962--973.

\bibitem{ZT2005}
{\sc G.~Zhou and K.~C. Toh}, {\em Superlinear convergence of a {N}ewton-type
  algorithm for monotone equations}, J. Optim. Theory Appl., 125 (2005),
  pp.~205--221.

\bibitem{ZL2008}
{\sc W.~J. Zhou and D.~H. Li}, {\em A globally convergent {BFGS} method for
  nonlinear monotone equations without any merit functions}, Math. Comp., 77
  (2008), pp.~2231--2240.

\end{thebibliography}

\end{document}